\documentclass[preprint]{elsarticle}

\usepackage[colorlinks=true,citecolor=black,linkcolor=black,urlcolor=blue]{hyperref}
\usepackage[italian, english]{babel}
\usepackage[utf8]{inputenc}
\usepackage{latexsym}
\usepackage{booktabs} 
\usepackage{amsfonts}
\usepackage{graphicx}
\usepackage{subfig}
\usepackage{caption}
\usepackage{textcomp} 
\usepackage{amsmath,amssymb,amsthm}
\usepackage{tikz,color,graphicx}
\usepackage{pgfplots}
\usepackage{floatflt,epsfig}

\usepackage{tikz}
\usetikzlibrary{positioning,chains,fit,shapes,calc}
\usetikzlibrary{arrows,intersections}

\numberwithin{equation}{section}

\theoremstyle{plain}
\newtheorem{teorema}{Theorem}[section]
\theoremstyle{definition}
\newtheorem{definizione}[teorema]{Definition}
\theoremstyle{definition}

\theoremstyle{definition}

\theoremstyle{definition}

\theoremstyle{plain}
\newtheorem{corollario}[teorema]{Corollary}
\theoremstyle{plain}
\newtheorem{proposizione}[teorema]{Proposition}
\theoremstyle{plain}
\newtheorem{lemma}[teorema]{Lemma}
\theoremstyle{plain}

\theoremstyle{plain}
\newtheorem{remark}[teorema]{Remark}
\theoremstyle{plain}

\def\dfrac#1#2{\lower0.15ex\hbox{\large$\frac{#1}{#2}$}}

\begin{document}

\pagestyle{plain}
\pagenumbering{arabic}

\title{Transition time asymptotics of queue-based\\ 
activation protocols in random-access networks \tnoteref{t1,t2}}
\tnotetext[t1]{Declarations of interest: none.}
\tnotetext[t2]{\emph{Funding:} This work was supported by the Netherlands Organisation for Scientific Research (NWO) [Gravitation Grant number 024.002.003--NETWORKS].}

\author[TUe,nokia]{S.C.~Borst}
\ead{s.c.borst@tue.nl}
\author[LU]{F.~den Hollander}
\ead{denholla@math.leidenuniv.nl}
\author[Firenze]{F.R.~Nardi}
\ead[url]{francescaromana.nardi@unifi.it}
\author[LU]{M.~Sfragara\corref{corresponding}}
\ead[url]{m.sfragara@math.leidenuniv.nl}

\cortext[corresponding]{Corresponding author}

\address[TUe]{Eindhoven University of Technology, Department of Mathematics and Computer Science, 5600 MB, Eindhoven, The Netherlands}
\address[nokia]{Nokia Bell Labs, 600 Mountain Avenue, Murray Hill, NJ, USA} 
\address[LU]{Leiden University, Mathematical Institute, 2300 RA, Leiden, The Netherlands} 
\address[Firenze]{University of Florence, Department of Mathematics and Computer Science ``Ulisse Dini", Viale Morgagni 67/a 50134, Florence, Italy}

\begin{abstract}

We consider networks where each node represents a server with a queue. An 
active node deactivates at unit rate. An inactive node activates at a rate that 
depends on its queue length, provided none of its neighbors is active.

For complete bipartite networks, in the limit as the queues become large, we 
compute the average transition time between the two states where one 
half of the network is active and the other half is inactive. We show that the
law of the transition time divided by its mean exhibits a trichotomy, 
depending on the activation rate functions.

\end{abstract}

\begin{keyword}

Random-access networks \sep activation protocols \sep
transition time \sep metastability

\emph{MSC2010:} 
60K25, 
60K30, 
90B15, 
90B18. 
\end{keyword}
\maketitle

\section{Introduction}
\label{S1}

Section~\ref{S1.1} provides motivation and background. Section~\ref{S1.2} formulates 
the mathematical model. Section~\ref{S1.3} states the main theorems. Section~\ref{S1.4} 
offers a brief discussion of these theorems, as well as an outline of the remainder of 
the paper.

 
\subsection{Motivation and background}
\label{S1.1}

In the present paper we investigate metastability properties and transition time 
asymptotics of \emph{queue-based random-access protocols in wireless networks}.
Specifically, we consider a stylized stochastic model for a wireless network (see 
Fig.~\ref{fig:network} below), represented in terms of an undirected graph $G = (N,B)$, 
referred to as the \emph{interference graph}. The set of nodes $N$ labels the servers and 
the set of bonds $B$ indicates which pairs of servers interfere and are therefore 
prevented from simultaneous activity. We denote by $X(t) \in \mathcal{X} $ the joint 
activity state at time $t$, with state space 
\begin{equation}
\label{Xdef}
\mathcal{X} = \big\{x \in \{0,1\}^N\colon\, x_ix_j = 0\,\,\,\forall\, (i,j) \in B\big\},
\end{equation}
where $x_i = 0$ means that node $i$ is inactive and $x_i =1$ that it is active. 

\begin{figure}[htbp]
\begin{center}
\vspace{0.5cm}
\includegraphics[width=.40\linewidth]{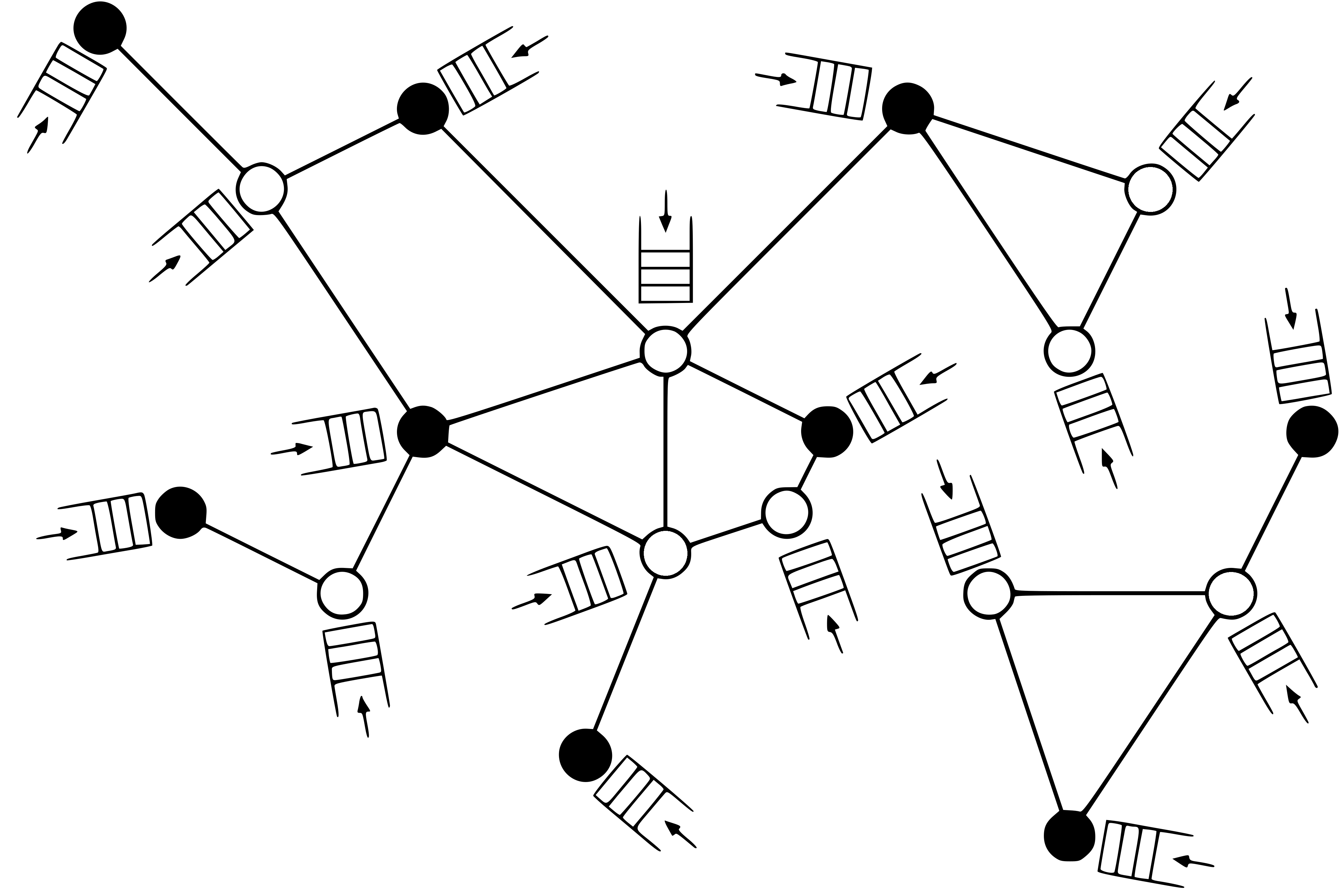}
\vspace{0.5cm}
\caption{\small A random-access network. Each node represents a server with a queue. 
Packets arrive that require a random service time.} 
\label{fig:network}
\end{center}
\vspace{-0.5cm}
\end{figure}

We assume that packets arrive at the nodes as independent Poisson processes 
and have independent exponentially distributed sizes. When a packet arrives at 
a node, it joins the queue at that node and the queue length undergoes an instantaneous 
jump equal to the size of the arriving packet. The queue decreases at a constant rate
$c$ (as long as it is positive) when the node is active. We denote by $Q(t) \in 
\mathbb{R}_+^N$ the joint queue size vector at time $t$, with $Q_i(t)$ representing 
the queue size at node $i$ at time $t$. When node $i$ is inactive at time $t$, it activates 
at a time-dependent exponential rate $r_i(Q_i(t))$, \emph{provided none of its neighbors 
is active}, where $q \mapsto r_i(q)$ is some increasing function. Activity durations are 
exponentially distributed with unit mean, i.e., when a node is active it deactivates at 
rate $1$. Note that $(X(t),Q(t))_{t \geq 0}$ evolves as a time-homogeneous 
Markov process with state space $\mathcal{X} \times \mathbb{R}_+^N$, since the transition 
rates depend on time only via the current state of the vector.

The above-described model has been thoroughly studied in the case
where the activation rate at each node~$i$ is fixed at some value~$r_i$,
$i = 1, \dots, N$.
In that case the joint activity process $(X(t))_{t \geq 0}$ behaves
as a reversible Markov process for any interference graph~$G$,
and has a product-form stationary distribution
\begin{equation}
\lim_{t \to \infty} \mathbb{P}\{X(t) = x\} =
Z_{\mathcal{X}}^{- 1}(r_1, \dots, r_N) \prod_{i = 1}^{N} r_i^{x_i}, \qquad
x \in \mathcal{X},
\label{prodstat1}
\end{equation}
with $Z_{\mathcal{X}}(r_1, \dots, r_N)$ denoting a normalization constant.
This basic model version was introduced in the eighties to analyze
the throughput performance of distributed resource sharing
and random-access schemes in packet radio networks, in particular the
so-called Carrier-Sense Multiple-Access (CSMA) protocol
\cite{BK80,BKMS87,Kelly85,KBC87,PY86,Yemini83}.
The model was rediscovered and further examined twenty years later
in the context of IEEE 802.11 (WiFi) networks \cite{DDT07,DT06,LKLW09,WK05}.

If we further restrict to $r_i \equiv r$ for all $i = 1, \dots, N$,
then the product-form distribution in~\eqref{prodstat1} simplifies to
\begin{equation}
\lim_{t \to \infty} \mathbb{P}\{X(t) = x\} = Z_{\mathcal{X}}^{- 1}(r)\,
r^{\sum_{i = 1}^{N} x_i}, \qquad x \in \mathcal{X},
\label{prodstat2}
\end{equation}
with $Z_{\mathcal{X}}(r) \equiv Z_{\mathcal{X}}(r, \dots, r)$.
From an interacting-particle systems perspective, the distribution
in~\eqref{prodstat2} may be recognized as the Gibbs measure
of a hard-core interaction model induced by the graph~$G$.
Hard-core interaction models are known to exhibit metastability effects,
where for certain graphs it takes an exceedingly long time for the
process to reach a stable state, starting from a metastable state
\cite{BdH05,NZB16,OV05}.
In particular, in a regime where the activation rate~$r$ grows large,
the stationary distribution of the joint activity process
in~\eqref{prodstat2} concentrates on states where the maximum number
of nodes is active, with possibly extremely slow transitions between them.

Metastability properties are not only of conceptual interest,
but also of great practical significance as the slow transitions
between activity states reflect spatial unfairness phenomena
and temporal starvation issues in random-access networks.
While the aggregate throughput may improve as the activation rate grows
large, individual nodes may experience prolonged periods of starvation,
possibly interspersed with long sequences of transmissions in rapid
succession, resulting in severe build-up of queues and long delays.
Indeed, the latter issues have been empirically observed in IEEE 802.11
networks, and have also been investigated through the lens of the
above-described model \cite{BBvL14,DT08,DTD09,GSK08,Zocca15}.
Gaining a deeper understanding of metastability properties and slow
transitions is thus instrumental in analyzing starvation behavior
in wireless networks, and ultimately of vital importance for designing
mechanisms to counter these effects and improve the overall
performance as experienced by users.

Note that slow transitions between activity states are not immediately
apparent from the stationary distribution in~\eqref{prodstat2}.
In fact, even when each node is active exactly the same fraction
of time in stationarity, it may well be the case that over finite time
intervals, certain nodes are basically barred from activity,
while other nodes are transmitting essentially all the time.
In other words, the stationary distribution is not directly
informative of the transition times between activity states that govern
the performance in terms of equitable transmission opportunities
for the various users during finite time windows.
Metastability properties that arise in the asymptotic regime where
the activation rate grows large, provide a powerful mathematical
paradigm to analyze the likelihood for such unfairness and starvation
issues to persist over long time periods.

The discussion above and the bulk of the literature pertain to the case
where the activation rates are fixed parameters.
As mentioned earlier, in the present paper we focus however
on \emph{queue-based} random-access protocols where the \emph{activation
rates are functions of the queue lengths at the various nodes}.
Specifically, the activation rate is an increasing function of the 
queue length of the node itself, and possibly a decreasing function
of the queue lengths of its neighbors, so as to provide greater
transmission opportunities to nodes with longer queues.
As a result, these rates vary over time as queues build up or drain
when packets are generated or transmitted.

Breakthrough work in \cite{GS10,JSSW10,RSS09,SS12} has shown that,
for suitable activation rate functions, queue-based random-access
schemes achieve maximum stability, i.e., provide stable queues
whenever feasible at all.
Thus these policies are capable of matching the optimal throughput
performance of centralized scheduling strategies, while requiring less
computation and operating in a distributed fashion.
On the downside, the very activation rate functions required
for ensuring maximum stability tend to result in long queues and poor
delay performance \cite{BBvL14,GBW14}.
This has sparked a strong interest in understanding, and possibly
improving, the delay performance of queue-based random-access schemes,
and analyzing metastability properties and transition times for the
joint activity process is a crucial endeavor in that regard.

When the activation rates are queue-dependent, the joint activity process
$(X(t))_{t \geq 0}$ may be viewed as a hard-core interaction model
with activation rates that depend on the additional queue state $Q(t)$.
The queue state not only depends on the history of the packet arrival
process (which causes upward jumps in the queue sizes), but also
on the past evolution of the activity process itself (through the
gradual reduction in queue sizes during activity periods).
The state-dependent nature of the activation rates raises interesting
and challenging issues from a methodological perspective,
and in particular requires novel concepts in order to handle
the two-way interaction between activity states and queue states.
We will specifcally examine the metastability properties
and transition times of the joint activity process in an asymptotic
regime where the initial queue sizes $Q_i(0)$, and hence the activation
rates $r_i(Q(t))$, $i = 1, \dots, N$, grow large in some suitable sense.

Throughout we focus on \emph{complete bipartite} interference graphs $G$: the 
node set can be partitioned into two nonempty sets $U$ and $V$ such that the bond set 
is the product of $U$ and $V$, i.e., two nodes interfere if and only if one belongs to $U$ 
and the other belongs to $V$. Thus, the collection of all independent sets of $G$ consists 
of all the subsets of $U$ and all the subsets of $V$.
While there is admittedly no specific physical reason
for focusing on complete bipartite interference graphs, this assumption
provides mathematical tractability and serves as a stepping stone
towards more general network topologies.
For convenience, we also assume that the activation rate functions are
of the form 
\begin{equation}
\label{rgdef}
r_i(t) = 
\left\{\begin{array}{ll} 
g_U(Q_i(t)), &i \in U, \\ 
g_V(Q_i(t)), &i \in V,
\end{array}
\right. 
\end{equation}
where $q \mapsto g_U(q)$ and $q \mapsto g_V(q)$ are non-decreasing functions such that $\lim_{q \to \infty} 
g_U(q) = \infty$, $ \lim_{q \to \infty} g_V(q)$ and $g_U(q) = g_V(q) = 0$ when $q < 0$. We denote by 
$u \in \mathcal{X}$ and $v \in \mathcal{X}$ the joint activity states where all the nodes in either $U$ or 
$V$ are active, respectively.

We will examine the distribution of the time until state $v$ is reached,
\begin{equation}
\label{Tvdef}
\tau_v = \inf\{t \geq 0\colon\, X(t) = v\},
\end{equation}
when the system starts from state $u$ at time $t=0$. We consider an asymptotic regime where 
the initial queue sizes $Q_i(0)$, $i \in U$, grow large in some suitable sense. As it turns out, 
the metastable behavior and asymptotic distribution of $\tau_v$ are closely related to those 
in a scenario where the \emph{activation rates are not governed by the random queue sizes}, 
but are \emph{deterministic} and of the form
\begin{equation}
\label{shdef}
r_i(t) = 
\left\{\begin{array}{ll} 
h_U(t), &i \in U, \\ 
h_V(t), &i \in V,
\end{array}
\right. 
\end{equation}
for suitable functions $t \mapsto h_U(t)$ and $t \mapsto h_V(t)$. 

Specifically, when the initial activity state is $u$ and the initial queue sizes $Q_i(0)$ are large 
for all $i \in U$, all the nodes in $U$ will initially be active virtually all the time, preventing any 
of the nodes in $V$ to become active. Consequently, the queue sizes of the nodes in $U$ will 
tend to decrease at rate $c - \rho_U  > 0$, while the queue sizes of the nodes in $V$ will tend 
to increase at rate $\rho_V > 0$, where $\rho_U$ and $\rho_V$ denote the common traffic intensity 
of the nodes in $U$ and  $V$, respectively. 

While the packet arrivals and activity periods are governed by random processes, the trajectories 
of the queue sizes will be \emph{roughly linear when viewed on the long time scales of interest}. 
This suggests that if we assume identical initial queue sizes $Q_i(0) \equiv Q_U(0)$, $i \in U$, 
and $Q_i(0) \equiv Q_V(0)$, $i \in V$, within the sets $U$ and $V$, respectively, then the asymptotic 
distribution of $\tau_v$ in \eqref{Tvdef} in the model with queue-dependent activation rates defined 
in \eqref{rgdef} should be close to that in the model with deterministic activation rates defined in 
\eqref{shdef} when we choose 
\begin{equation}
h_U(t) = g_U\big(Q_U(0) - (c- \rho_U ) t\big), \qquad h_V(t) = g_V\big(Q_V(0) + \rho_V  t\big).
\end{equation}
The asymptotic distribution of $\tau_v$ in the latter scenario was characterised in \cite{BdHNT17}, 
with the help of the metastability analysis for hard-core interaction models developed in \cite{dHNT17}.


\subsection{Mathematical model}
\label{S1.2}

We consider the case where $G=(N,B)$ is a \emph{complete bipartite} graph, i.e., 
$N = U \cup V$ and $B$ is the set of all bonds that connect a node in $U$ to a node in 
$V$ (see Fig.~\ref{fig:completebipartitegraph}).

\begin{figure}[htbp]
\begin{center}
\vspace{0.5cm}
\includegraphics[width=.18\linewidth]{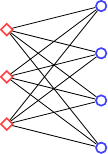}
\vspace{0.5cm}
\caption{\small A complete bipartite graph with $|U| = 3$ and $|V| = 4$. At time $t=0$, square-shaped 
nodes are active and circle-shaped nodes are inactive.} 
\label{fig:completebipartitegraph}
\end{center}
\end{figure}

\begin{definizione}[{\bf State of a node}]
A node in the network can be \textit{active} or \textit{inactive}. The \textit{state of node} $i$ 
at time $t$ is described by a Bernoulli random variable $X_i(t) \in \{0,1\}$, defined as
\begin{equation} 
X_i(t) = 
\begin{cases} 
0, \text{ if } i \text{ is inactive at time } t ,\\ 
1, \text{ if } i \text{ is active at time } t. 
\end{cases}
\end{equation}
The joint activity state $X(t)$ at time $t$ is an element of the set $\mathcal{X}$ defined in 
\eqref{Xdef}: the feasible configurations of the network correspond to the collection of 
independent sets of $G$. We denote by $u \in \mathcal{X} \,\, (v \in \mathcal{X})$
the configuration where all the nodes in $U$ are active (inactive) and all the nodes in $V$ 
are inactive (active).
\hfill\qed 
\end{definizione}

\begin{definizione}[{\bf Pre-transition time and transition time}]
The following two times are the main objects of interest in the present paper.
\begin{itemize}
\item
The \textit{pre-transition time} $\bar{\tau}_v$ is the first time a node in $V$ turns active, i.e.,
\begin{equation} \label{pretransition}
\bar{\tau}_v = \inf \{ t > 0\colon\, X_i(t) = 1\,\, \exists\, i \in V \}.
\end{equation}
\item
The \textit{transition time} $\tau_v$ is the first time configuration $v$ is hit, i.e.,
\begin{equation} \label{transition}
\tau_v = \min \big\{t \geq 0\colon\, X_i(t) = 0 \,\,\forall\, i \in U , \,X_i(t) = 1 \,\,\forall\, i \in V \big\}.
\end{equation}
\end{itemize}
\hfill\qed
\end{definizione}

\noindent
We are interested in the distribution of $\bar{\tau}_v$ and $\tau_v$ given that $X(0)=u$. The 
pre-transition time plays an important role in our analysis of the transition time, because the 
evolution of the network is simpler on the interval $[0,\bar{\tau}_v]$ than on the interval 
$[\bar{\tau}_v,\tau_v]$. However, we will see that $\tau_v-\bar{\tau}_v \ll \bar{\tau}_v$ when
the initial queue lengths are large, so that both times have the same asymptotic scaling behavior. See Fig.~\ref{pretrfig} for a representation of the pre-transition configuration.

\medskip\noindent
\begin{figure}[htbp]
\begin{minipage}{0.26\linewidth}
\includegraphics[keepaspectratio=true,width=0.7\textwidth]{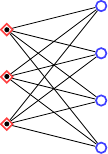}
\end{minipage}
\hfill
\begin{minipage}{0.26\linewidth}
\includegraphics[keepaspectratio=true,width=0.7\textwidth]{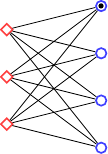}
\end{minipage}
\hfill
\begin{minipage}{0.26\linewidth}
\includegraphics[keepaspectratio=true,width=0.7\textwidth]{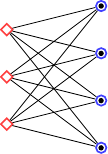}
\end{minipage}
\vspace{0.5cm}
\caption{\small \emph{Left}: initial configuration $u$. \emph{Center}: pre-transition configuration. 
\emph{Right}: final configuration $v$.}
\label{pretrfig}
\end{figure}

An active node $i$ turns inactive according to a deactivation Poisson clock: when the clock ticks, 
the node switches itself off. Vice versa, an inactive node $i$ attempts to become active at the 
ticks of an activation Poisson clock: an attempt at time $t$ is successful when no neighbors 
of $i$ are active at time $t^-$. Different models can be studied depending on the choice of the 
activation and deactivation rates of the clocks. Models where these rates are deterministic 
functions of $t$ are called \textit{external models}. In the present paper we are interested in 
what are called \textit{internal models}, where the clock rates at node $i$ depend on the queue 
length at node $i$ at time $t$.

\begin{definizione}[{\bf Queue length at a node}]
Let $t \mapsto Q_i^+(t)$ be the \textit{input process} describing packets arriving according to 
a Poisson process $t \mapsto N(t) = \mathrm{Poisson}(\lambda t)$ and having i.i.d.\ exponential 
service times $Y_j \simeq \mathrm{Exp}(\mu)$, $j \in \mathbb{N}$. Let $t \mapsto Q_i^-(t)$ be the
\textit{output process} representing the cumulative amount of work that is processed in the time 
interval $[0,t]$ at rate $c$, which equals $cT_i(t) = c \int_0^t X_i(s) ds$. Define
\begin{equation} 
\Delta_i(t) = Q_i^+(t) - Q_i^-(t) = \sum_{j=0}^{N_i(t)} Y_{ij} - c T_i(t)
\end{equation}
and let $s^* = s^*(t)$ be the value where $\sup_{s \in [0,t]} [\Delta_i(t) - \Delta_i(s)]$ is reached, i.e., 
equals $[\Delta_i(t) - \Delta_i(s^*-)]$. Let $Q_i(t) \in \mathbb{R}_{\geq 0}$ denote the queue length 
at node $i$ at time $t$. Then
\begin{equation}
Q_i(t) = \max\big\{ Q_i(0) + \Delta_i(t),\,\Delta_i(t)-\Delta_i(s^*-) \big\},
\end{equation}
where $Q_i(0)$ is the \textit{initial queue length}. The maximum is achieved by the first term when 
$Q_i(0) \geq -\Delta_i(s^*-)$ (the queue length never sojourns at $0$), and by the second term 
when $Q_i(0) < -\Delta_i(s^*-)$ (the queue length sojourns at $0$ at time $s^*-$).
\hfill\qed
\end{definizione}

\noindent
In order to ensure that the queue length remains non-negative, we let a node switch itself off 
when its queue length hits zero. The initial queue lengths are assumed to be
\begin{equation}
\label{initialqueues}
Q_i(0) = 
\left\{\begin{array}{ll} 
\gamma_U r, &i \in U, \\  
\gamma_V r, &i \in V, 
\end{array}
\right.
\end{equation}
where $\gamma_U \geq \gamma_V > 0$, and $r$ is a parameter that tends to infinity. Thus, the initial 
queue lengths are of order $r$, i.e., $Q_i(0) \asymp r$, and the ones at the nodes in $U$ are larger than 
the ones at the nodes in $V$. Note that both the pre-transition and the transition time grow to infinity 
with $r$, since the larger the initial queue lengths are, the longer it takes for the transition to occur.

For each node $i$, the \emph{input process} $t \mapsto Q_i^+(t) = \sum_{j=0}^{N_i(t)} Y_{ij}$ is a 
compound Poisson process. In the time interval $[0,t]$ packets arrive according to a Poisson process 
$t \mapsto N_i(t)$ with a rate $\lambda_U$ or $\lambda_V$, depending on whether the node is in 
$U$ or $V$. Moreover, each packet $j$ brings the information of its service time: the service time 
$Y_{ij}$ of the $j$-th packet at node $i$ is exponentially distributed with parameter $\mu$. Hence 
the expected value of $Q_i^+(t)$ for a node in $U$ is the product of the expected value $\mathbb{E}
[N_i(t)] = \lambda_U t$ and the expected value $\mathbb{E}[Y_j] = 1/\mu$, i.e., $\mathbb{E}[Q_i^+(t)] 
= (\lambda_U / \mu)t = \rho_U t$. Analogously, for a node in $V$ we have $\mathbb{E}[Q_i^+(t)] 
= \rho_V t$. We assume that all the service times are i.i.d.\ random variables, and are independent of 
the Poisson process $t \mapsto N_i(t)$.

For each node $i$, the \emph{output process} is $t \mapsto Q_i^-(t) = c T_i(t) = c \int_0^t X_i(u)du$, 
where the \textit{activity process} $t \mapsto T_i(t)$ represents the cumulative amount of active time 
of node $i$ in the time interval $[0,t]$. This is not independent of the input process. Intuitively, the 
average queue length increases when the node is inactive and decreases when the node is active, 
which means that packets are being served at a rate $c$ larger than their arrival rate, i.e., $c > \rho_U, 
\rho_V > 0$. Since all nodes in $V$ are initially inactive, for some time the queue length of these nodes 
in $V$ is not affected by their output process. However, as soon as a vertex in $V$ turns active, we 
have to consider its output process as well.

The choice of functions $g_U, g_V$ in \eqref{rgdef} determines the transition time of 
the network, since the activation rates of the nodes depend on them. We will assume that $g_U,g_V$ 
fall in the following class of functions:
\begin{equation}
\mathcal{G} = \Big\{g\colon\,\mathbb{R} \to \mathbb{R}_{\geq 0}\colon\,
g \text{ non-decreasing and continuous},
 g(\mathbb{R}_{\leq 0})=0, \lim_{x \to \infty} g(x) = \infty\Big\}. 
\end{equation}

\begin{definizione}[{\bf Models}] 
Let $g_U,g_V \in \mathcal{G}$ and $\delta > 0$. Assume (\ref{initialqueues}). The four models of interest 
in the present paper
are the following:
\begin{itemize}
\item 
In the \textit{internal model} the deactivation Poisson clocks tick at rate 1, while the activation 
Poisson clocks tick at rate 
\begin{equation}
\label{rint}
r_i^{\mathrm{int}}(t) = 
\left\{\begin{array}{ll} 
g_U(Q_i(t)), &i \in U, \\ 
g_V(Q_i(t)), &i \in V,
\end{array}
\right. 
\qquad t \geq  0.
\end{equation}
\item
In the \textit{external model} the deactivation Poisson clocks tick at rate 1, while the activation Poisson 
clocks tick at rate 
\begin{equation}
\label{rext}
r_i^{\mathrm{ext}}(t) = 
\left\{\begin{array}{ll} 
g_U(\gamma_U r - (c - \rho_U)t ), &i \in U, \\ 
g_V(\gamma_V r + \rho_V t), &i \in V, 
\end{array}
\right.
\qquad t \geq 0.
\end{equation}
\item
In the \textit{lower external model} the deactivation Poisson clocks tick at rate 1, while the activation 
Poisson clocks tick at rate 
\begin{equation}
\label{rlow}
r_i^{\text{low}}(t) = 
\left\{\begin{array}{ll} 
g_U(\gamma_U r - (c- \rho_U )t - \delta r ), &i \in U, \\ 
g_V(\gamma_V r + \rho_Vt + \delta r), &i \in V,
\end{array}
\right. 
\qquad t \geq 0.
\end{equation}
\item
In the \textit{upper external model} the deactivation Poisson clocks tick at rate 1, while the activation 
Poisson clocks tick at rate 
\begin{equation}
\label{rupp}
r_i^{\text{upp}}(t) = 
\left\{\begin{array}{ll} 
g_U(\gamma_U r - (c- \rho_U)t + 2 \delta r ), &i \in U, \\ 
g_V(\gamma_V r + \rho_Vt  - \delta r), &i \in V,
\end{array}
\right. 
\qquad t \geq 0.
\end{equation}
\end{itemize}
\hfill\qed
\end{definizione}

\noindent
Note that in the external models the rates depend on time via certain fixed parameters, while in the 
internal model the rates depend on time via the actual queue lengths at the nodes. In the lower 
external model the activation rates in $U$ tend to be \emph{less aggressive} than in the internal model 
(i.e., the activation clocks tick less frequently), while the activation rates in $V$ tend to be \emph{more 
aggressive}. In the upper external model the reverse is true: the activation rates in $U$ are more 
aggressive and the activation rates in $V$ are less aggressive. For simplicity, when considering 
the external model we sometimes write 
\begin{equation}
r_U(t) \text{ and } r_V(t)
\end{equation}
for the activation rate at time $t$ of a node in $U$ and a node in $V$. We will see that the upper 
external model is actually defined only for $t \in [0,T_U]$ with $T_U = \frac{\gamma_U}{c-\rho_U} r$ 
(see Section~\ref{S2} for details). However, the transition occurs with high probability before time
$T_U$.


\subsection{Main theorems}
\label{S1.3}

The main goal of the present paper is to compare the transition time of the internal model with that 
of the two external models. Through a large-deviation analysis of the queue length process at each 
of the nodes, we define a notion of \emph{good behavior} that allows us to define perturbed models 
with externally driven rates that \emph{sandwich} the queue lengths of the internal model and its 
transition time. We show with the help of \emph{coupling} that with high probability the asymptotic 
behavior of the mean transition time for the internal model is the same as for the external model. 

The metastable behavior and the transition time $\tau_v$ of a network in which the activation 
rates are time-dependent in a deterministic way was characterized in \cite{BdHNT17}, and the 
asymptotic distribution of $\tau_v$ was studied in detail. For $s \geq 0$, let 
\begin{equation} \label{nus}
\nu(s) = \frac{1}{\mathbb{E}_{u} [ \tau_v](s)}
\end{equation}
be the inverse mean transition time of the time-homogeneous model where we freeze the the activation rates $r_U$ and $r_V$ at time $s$, i.e., we consider the model with constant rates
\begin{equation}
r_i^{\text{ext}}(t) =
\left\{\begin{array}{ll} 
r_U(s),& i \in U, \\
r_V(s),& i \in V,
\end{array}
\right.
\qquad t \geq 0.
\end{equation}
Then, for any time scale $M=M(r)$ and any threshold $x \in [0, \infty)$,
\begin{equation}
\label{siamakpaper2}
\lim_{r \to \infty} \mathbb{P}_u \bigg(\frac{\tau_v}{M} > x \bigg) = 
\left\{
\begin{array}{ll}
0, &\text{ if } M \nu(Mx) \succ 1,\\
e^{-\int_0^xM \nu(Ms) ds}, &\text{ if } M \nu(Mx) \asymp 1,\\
1, &\text{ if } M \nu(Mx) \prec 1,
\end{array} 
\right.
\end{equation}
where $a \succ b$ means $b=o(a)$, $a \prec b$ means $a=o(b)$ and $a \asymp b$ means $a = \Theta(b)$. If we let $M_c$ be the unique solution of the equation
\begin{equation} 
\label{Mc}
M\nu(M)=1,
\end{equation}
then the transition occurs on the time scale $M_c$, in the sense that $\mathbb{P}_u(\tau_v > t) \approx 1$ 
for $t \prec M_c$ and $\mathbb{P}_u(\tau_v > t) \approx 0$ for $t \succ M_c$. On the critical time scale 
$M_c$, the transition time follows an exponential law with time-varying rate. It was proven in \cite{dHNT17} 
that, for a complete bipartite graph and $s \in [0, \infty)$,
\begin{equation} 
\label{meantaucomplete}
\mathbb{E}_u[\tau_v](s) = \frac{1}{|U|}\, r_U(s)^{|U|-1}\, [1+o(1)], \qquad r \to \infty.
\end{equation}

\begin{remark}[{\bf Scaling}]
\label{regimes}
{\rm In \cite[Theorem 1.3]{BdHNT17} a formula different from \eqref{siamakpaper2} was proved, 
which involves three regimes and is valid for a large class of bipartite graphs. For complete bipartite graphs, 
however, this formula produces the same asymptotics as \eqref{siamakpaper2}. We will return to this point in Section~\ref{S4}.}
\end{remark}

We want the nodes in $V$ to be more aggressive than the nodes in $U$, so that the transition from $u$ 
to $v$ can be viewed as the crossover from a ``metastable state'' to a ``stable state''. Therefore we assume from now on that 
\begin{equation} \label{aggress}
\lim_{x \to \infty} \frac{g_V(x)}{g_U(x)} = \infty,
\end{equation}
and we focus on activation rates for nodes in $U$ of the form $g_U(x) \sim G x^{\beta}$ as 
$x \to \infty$ (see Remark \ref{remarksvf} in Section \ref{S4} for more general activation rates 
$g_U$). We do not require any further assumption on the activation rates $g_V$ and we will see that, when \eqref{aggress} is satisfied, the asymptotic distribution of the transition time as $r \to \infty$ is independent of $g_V$. The following two theorems will be proven in Sections~\ref{sec4.1}--\ref{sec4.2} with 
the help of \eqref{nus}--\eqref{meantaucomplete}.

\begin{teorema}[{\bf Critical time scale in the external model}]
\label{thmMc}
Suppose that $g_U(x) \sim G x^{\beta}$ with $\beta,G \in (0,\infty)$. Then
\begin{equation}
M_c = F_c r^{1 \vee \beta (|U|-1)} \, [1+o(1)], \qquad r \to \infty,
\end{equation} 
with 
\begin{equation}
\label{factorFc}
F_c  = \left\{\begin{array}{ll} 
\frac{\gamma_U^{\beta(|U|-1)}}{|U|G^{-(|U|-1)}}, &\text{ if } \beta \in (0, \frac{1}{|U|-1}),\\[0.2cm]
\frac{\gamma_U}{|U|G^{-(|U|-1)} + (c-\rho_U)}, &\text{ if } \beta = \frac{1}{|U|-1},\\[0.2cm]
\frac{\gamma_U}{c-\rho_U}, &\text{ if } \beta = (\frac{1}{|U|-1},\infty).
\end{array}
\right.
\end{equation} 
\end{teorema}

\begin{teorema}[{\bf Transition time in the external model}] 
\label{thmexternal}
Suppose that $g_U(x) = G x^{\beta}$ with $\beta,G \in (0,\infty)$. Then
\begin{equation}
\label{scalingext}
\mathbb{E}_{u} [\tau_v^{\mathrm{ext}}] = F_c r^{1 \vee \beta(|U|-1)}\, [1+o(1)],  \qquad r \to \infty.
\end{equation} 
and
\begin{equation}
\lim_{r \to \infty} \mathbb{P}_u \bigg( \frac{\tau_v^{\mathrm{ext}}}
{\mathbb{E}_u[\tau_v^{\mathrm{ext}}]} > x \bigg) 
= \mathcal{P}(x), \qquad x \in [0,\infty),
\end{equation}
with
\begin{equation}
\label{Plaw}
\mathcal{P}(x) 
= \left\{\begin{array}{ll}
e^{-x}, &\text{ if } \beta \in (0, \frac{1}{|U|-1}),\, x \in [0,\infty),\\[0.2cm]
(1-C x)^{\frac{1-C}{C}}, &\text{ if } \beta = \frac{1}{|U|-1},\, x \in [0, \frac{1}{C}), \\[0.2cm]
0,  & \,\, \text {if } \beta = \frac{1}{|U|-1},\, x \in [\frac{1}{C}, \infty), \\[0.2cm]
1,  & \,\, \text {if } \beta \in (\frac{1}{|U|-1},\infty),\, x \in [0,1), \\[0.2cm]
0,  & \,\, \text {if } \beta \in (\frac{1}{|U|-1},\infty),\, x \in [1,\infty),
\end{array}
\right.
\end{equation}
and $C = \frac{F_c(c-\rho_U)}{\gamma_U} \in (0,1)$.
\end{teorema}

\noindent
In other words, the mean transition time scales like $M_c$, while the distribution of the transition
time divided by its mean is exponential, truncated polynomial, respectively, deterministic (see 
Fig.~\ref{fig:dichotomy}).  

\begin{figure}[htbp]
\vspace{0.4cm}
\begin{center}
\setlength{\unitlength}{0.3cm}
\begin{picture}(10,8)(18,0)
{\thicklines
\qbezier(0,-1)(0,3)(0,7)
\qbezier(-1,0)(6,0)(12,0)
\qbezier(0.1,5)(3,0.2)(10,0.2)
}
\put(12.5,-.25){$x$}
\put(-.8,7.6){$\mathcal{P}(x)$}
{\thicklines
\qbezier(17,-1)(17,3)(17,7)
\qbezier(16,0)(23,0)(29,0)
\qbezier(17.1,5)(19,0.2)(23,0)
\qbezier(23,0)(25,0)(27,0)
}
\qbezier[40](17.1,5)(21,5)(23,0)
\put(29.5,-.25){$x$}
\put(16.2,7.6){$\mathcal{P}(x)$}
\put(22.5,-1.7){$\tfrac{1}{C}$}
\put(23,0){\circle*{0.35}}
{\thicklines
\qbezier(34,-1)(34,3)(34,7)
\qbezier(33,0)(40,0)(46,0)
\qbezier(34.1,5)(37,5)(40,5)
\qbezier(40,5)(40,2.5)(40,0)
}
\put(46.5,-.25){$x$}
\put(33.2,7.6){$\mathcal{P}(x)$}
\put(39.7,-1.5){$1$}
\put(40,0){\circle*{0.35}}

\end{picture}
\end{center}
\vspace{0.5cm}
\caption{\small Trichotomy for $x \mapsto \mathcal{P}(x)$: $\beta \in (0, \frac{1}{|U|-1}]$ (left);
$\beta = \frac{1}{|U|-1}$ (middle); $\beta \in (\frac{1}{|U|-1},\infty)$ (right). The curve in the 
middle is convex when $C \in (0,1/2)$ and concave when $C \in (1/2,1)$. The curve 
on the right is the limit of the curve in the middle as $C \uparrow 1$.}
\label{fig:dichotomy}
\vspace{0.2cm}
\end{figure}
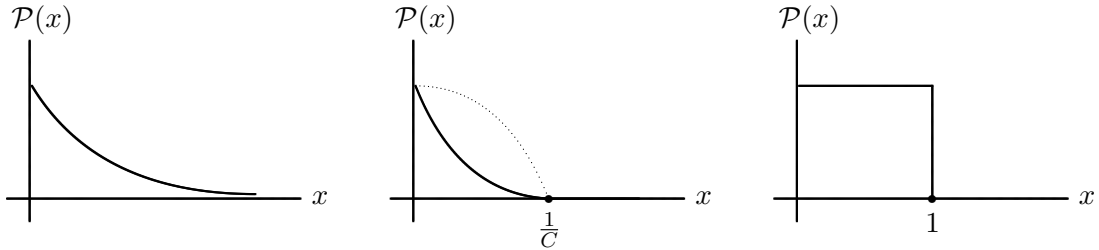
\noindent
As shown in Remark~\ref{remarksvf}, we can even include the case $\beta=0$, and get that 
if $g_U(x) =  \hat{\mathcal{L}}(x)$ with $\lim_{x\to\infty} \hat{\mathcal{L}}(x)=\infty$, then 
\begin{equation}
\mathbb{E}_{u} [ \tau_v^{\mathrm{ext}}]  = M_c\,[1+o(1)], \qquad 
M_c = \frac{1}{|U|}  \hat{\mathcal{L}}(\gamma_U r)^{|U|-1}\,[1+o(1)],  
\qquad r \to \infty,
\end{equation}
and $\mathcal{P}(x)=e^{-x}$, $x \in [0,\infty)$. Similar properties hold for the lower and the 
upper external model, with perturbed $F_{c,\delta}^{\text{low}}$ and $F_{c,\delta}^{\text{upp}}$ 
satisfying
\begin{equation}
\label{effe}
\lim_{\delta \downarrow 0}F_{c,\delta}^{\text{low}}
= \lim_{\delta \downarrow 0}F_{c,\delta}^{\text{upp}} = F_c.
\end{equation}

The main result in the present paper is the following sandwich of $\tau_v^{\text{int}}$ between 
$\tau_v^{\text{low}}$ and $\tau_v^{\text{upp}}$, for which we already know the asymptotic behavior. 
Because of this sandwich we can deduce the asymptotics of the transition time in the internal model.

\begin{teorema}[{\bf Transition time in the internal model}]
\label{thminternal}
For $\delta> 0$ small enough, there exists a coupling such that
\begin{equation}
\label{coup}
\lim_{r \to \infty} \mathbb{\hat{P}}_u \big(\tau_v^{\mathrm{low}} 
\leq \tau_v^{\mathrm{int}} \leq \tau_v^{\mathrm{upp}} \big) = 1,
\end{equation}
where $\mathbb{\hat{P}}_u$ is the joint law induced by the coupling, with all three models starting 
from the configuration $u$. Consequently, if $g_U(x) \sim G x^{\beta}$ with $\beta,G \in (0,\infty)$, 
then 
\begin{equation}
\label{meanasymp}
\mathbb{E}_{u} [ \tau_v^{\mathrm{int}}] = F_cr^{1 \vee \beta (|U|-1)} \, [1+o(1)], \qquad r \to \infty,
\end{equation} 
and
\begin{equation}
\lim_{r \to \infty} \mathbb{P}_u \bigg( \frac{\tau_v^{\mathrm{int}}}
{\mathbb{E}_u[\tau_v^{\mathrm{int}}]} > x \bigg) = \mathcal{P}(x), \qquad x  \in [0,\infty).
\end{equation}
\end{teorema}


\subsection{Discussion and outline}
\label{S1.4}

{\bf Theorems.} 
Theorem~\ref{thmexternal} gives the \emph{leading-order asymptotics} of the transition time 
in the external model, including the lower and the upper external model. Theorem~\ref{thminternal} 
is the main result of our paper and provides the \emph{leading-order asymptotics} of the transition 
time in the internal model, via the coupling in \eqref{coup} and the continuity property in \eqref{effe}. 
Equations (\ref{factorFc})--(\ref{scalingext}) identify the scaling of the transition time in terms of the 
model parameters. The trichotomy between $\beta \in (0, \frac{1}{|U|-1})$, $\beta = \frac{1}{|U|-1}$ 
and $\beta \in (\frac{1}{|U|-1},\infty)$ is particularly interesting, and leads to different limit laws for 
the transition time on the scale of its mean. 

\medskip\noindent
{\bf Interpretation of trichotomy.}
In order to interpret the above trichotomy, observe first of all that the activation rates of each of 
the nodes in $U$ remain of order $r^\beta$ almost all the way up $T_U$. Specifically, in the 
absence of the nodes in $V$, by time $y T_U$, $y \in [0,1)$, the queue lengths of the nodes 
in $U$ have decreased by roughly a fraction $y$, and their activation rates are approximately 
$G(1-y)^\beta r^\beta$. Hence the fraction of joint inactivity time of the nodes in $U$ is of order 
$(1/r^\beta)^{|U|}=r^{- \beta |U|}$, and all nodes in $U$ are simultaneously inactive for the first 
time after a period of order $r^{-\beta}/r^{-\beta|U|} = r^{\beta(|U|-1)}$, which is $o(r)$ when 
$\beta < \frac{1}{|U|-1}$. With the nodes in $V$ actually present, these then all activate and 
the transition occurs almost immediately with high probability (see Section~\ref{sec4.3}). Note 
that the queue lengths of the nodes in $U$ have only decreased by an amount of order 
$r^{\beta(|U|-1)} = o(r)$, and hence are still of order $r$. In contrast, when $\beta = \frac{1}{|U|-1}$, 
the probability that all nodes in $U$ become simultaneously inactive before time $y T_U$ is 
approximately $\pi(y)$ with $\pi(y) = 1-(1-y)^{(1-C)/C}$, $y \in [0,1)$ (see \eqref{Plaw}). Again, 
the nodes in $V$ then all activate and the transition occurs almost immediately with high 
probability. Note that the queue lengths in the nodes in $U$ have then dropped by a non-negligible 
fraction, but are still of order $r$. A potential scenario is that the nodes in $U$ are not all 
simultaneously inactive until their activation rates have become of a smaller order than 
$r^{\beta}$, due to the queue lengths no longer being of order $r$ just before time $T_U$. 
However, the fact that $\pi(y) \uparrow 1$ as $y \uparrow 1$ implies that this scenario has 
negligible probability in the limit. In contrast, this scenario does occur when $\beta > 
\frac{1}{|U|-1}$, implying that the crossover occurs in a narrow window around $T_U$ (see 
Sections~\ref{sec4.1}--\ref{sec4.2} for details). We will see that this window has size 
$O(r^{1/\beta(|U|-1)}) = o(r)$. In particular, the \emph{window gets narrower as the activation rate 
for nodes in $U$ increases}.

\medskip\noindent
{\bf Proofs.} 
We look at a single-node queue length process $t \mapsto Q(t)$ and prove that with high 
probability it follows a path that lies in a \emph{narrow tube around its mean path} (see
Fig.~\ref{fig:tubes}). We study separately the input process $t \mapsto Q^+(t)$ and the 
output process $t \mapsto Q^-(t)$: we use Mogulskii's theorem (a pathwise large-deviation 
principle) for the first, and Cram\'er's theorem (a pointwise large-deviation principle) for the 
second. We derive upper and lower bounds for the queue length process and we use these 
bounds to construct two couplings that allow us to compare the different models. 

\begin{center}
\begin{figure}[htbp]
\begin{tikzpicture}[
scale=0.60,
thick,
>=stealth',
dot/.style = {
draw,
fill = white,
circle,
inner sep = 0pt,
minimum size = 4pt
}
]
\coordinate (O) at (0,0);
\draw[->] (-0.3,0) -- (10,0) coordinate[label = {below:$t$}] (xmax);
\draw[->] (0,-0.3) -- (0,5.5) coordinate[label = {left:$Q_i(t)$}] (ymax);
\path[name path=x] (0,3) -- (6,0);  
\path[name path=low] (0,2.5) -- (5.5,0);
\path[name path=upp] (0,3.5) -- (6.5,0);
\draw      (0,2.5) -- (5,0) node[pos=0.5, below left] {LB};
\draw      (0,4.5) -- (7,1) node[pos=0.5, above right] {UB};
\draw  [dashed]    (0,3.5) -- (7,0) node[pos=0.5, above right] {};
\node[left, ] at (0,3.5) {$\gamma_U r$};
\draw[ <->] (1,3.05) --(1,3.95);
\node[right, ] at (1,3.35) {$\delta r$};
\node[above] at (7,4) {slope = $c-\rho_U$};
\node[below] at (7,0) {$T_U$};
\end{tikzpicture}
\hspace{0.5cm}
\begin{tikzpicture}[
scale=0.60,
thick,
>=stealth',
dot/.style = {
draw,
fill = white,
circle,
inner sep = 0pt,
minimum size = 4pt
}
]
\coordinate (O) at (0,0);
\draw[->] (-0.3,0) -- (10,0) coordinate[label = {below:$t$}] (xmax);
\draw[->] (0,-0.3) -- (0,5.5) coordinate[label = {left:$Q_j(t)$}] (ymax);
\path[name path=x] (0,3) -- (6,0);  
\path[name path=low] (0,2.5) -- (5.5,0);
\path[name path=upp] (0,3.5) -- (6.5,0);
\draw      (0,2.5) -- (6,5.5) node[pos=0.5, above left] {UB};
\draw    [dashed]  (6,5.5) -- (7.2,6.1) ;
\draw      (0,0.5) -- (6,3.5) node[pos=0.5,below right] {LB};
\draw   [dashed]   (6,3.5) -- (7.2,4.1) ;
\draw  [dashed]    (0,1.5) -- (6,4.5) node[pos=0.5, above right] {};
\draw  [dashed]    (6,4.5) -- (7.2,5.1)  {};
\node[above] at (7,1) {slope = $\rho_V$};
\node[left, ] at (0,1.5) {$\gamma_V r$};
\draw[ <->] (1,2.95) -- (1,2.05);
\node[right, ] at (1,2.65) {$\delta r$};
\end{tikzpicture}
\caption{\small Sketches of the tubes around the mean of the queue length processes, 
respectively, for a node $i \in U$ and a node $j \in V$.} 
\label{fig:tubes}
\end{figure}
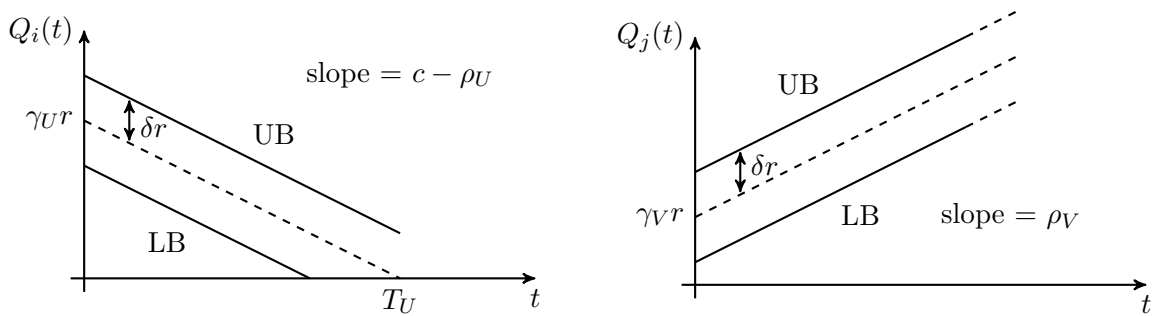
\vspace{-1cm}
\end{center}

\medskip\noindent
{\bf Dependent packet arrivals.} 
Our large-deviation estimates are so sharp that we can actually allow the Poisson processes 
of packet arrivals at the different nodes to be \emph{dependent}. Indeed, as long at the 
marginal processes are Poisson, our large-deviation estimates are valid at every single 
node, and since the network is finite a simple union bound shows that they are also 
valid for all nodes simultaneously, at the expense of a negligible factor that is equal to
the number of nodes. For modeling purposes independent arrivals are natural, but it is
interesting to allow for dependent arrivals when we want to study activation protocols that 
are more involved.  

\medskip\noindent
{\bf Open problems.} 
If we want to understand how small the term $o(1)$ in \eqref{meanasymp} actually is, then we 
need to derive sharper estimates in the coupling. One possibility would be to study moderate 
deviations for the queue length processes and to look at shrinking tubes. We do not pursue 
such refinements here. Our main focus for the future will be to extend the model to more complicated 
settings, where the activation rate at node $i$ depends also on the queue length at the neighboring 
nodes of $i$. We want to be able to compare models with (externally driven) time-dependent rates 
and models with (internally driven) queue-dependent rates, and show again that their metastable 
behavior is similar. We also want to move away from the complete bipartite interference graph 
and consider more general graphs that capture more realistic wireless networks.

\medskip\noindent
{\bf Other models.} 
There are other ways to define an internal model. We mention a few examples.
\begin{itemize}
\item[(i)]
A simple variant of our model is obtained by fixing the activation rates, but letting the rate 
at time $t$ of the Poisson deactivation clock of node $i$ depend on the reciprocal of the 
queue length at time $t$, i.e., $1/g_i(Q_i(t))$ for some $g_i \in \mathcal{G}$. This can be 
equivalently seen as a unit-rate Poisson deactivation clock, where node $i$ either deactivates 
with a probability reciprocal to $g_i(Q_i(t))$, or starts a second activity period. Nodes with 
a large queue length are more likely to remain active for a long time before switching off, 
while nodes with a short queue length have extremely short activity periods. If at time $t$ 
the activation clock of an inactive node with $Q_i(t) = 0$ ticks, then the node does not become 
active. On the other hand, if during an activity period the queue length of an active node hits 
zero, then the server switches itself off independently of the deactivation rate. For fixed 
activation and deactivation rates, this model and our internal model are equivalent up to a 
time scaling factor. In particular, they have similar stationary distributions.
\item[(ii)]
An alternative approach is to use a discrete notion of queue length, namely, $Q_i(t) = N_i(t) 
- S_i(t)$, where $N_i(t)$ is a Poisson process with rate $\lambda$, denoting the number of 
packets arriving at node $i$ during $[0,t]$, while $S_i(t)$ indicates the total number of times 
node $i$ turns active (or inactive) during $[0,t]$ (we may use $\lambda_U$ and $\lambda_V$ 
to represent different arrival rates for the two sets $U$ and $V$). The processes $t \mapsto S_i(t)$ 
and $t \mapsto N_i(t)$ are assumed to be independent. We can define a model where each 
time a node turns active it serves exactly one packet and then switches off again. The activation 
clocks still have rates $g_i(Q_i(t))$ with $g_i \in \mathcal{G}$. We can establish results similar 
to our internal model by adapting the arguments to the discrete setting.
\end{itemize}

\medskip\noindent
{\bf Outline.} 
The remainder of the paper is organized as follows. In Section~\ref{S2} we state large-deviation
bounds for the input and the output process, which allow us to show that the queue length process 
at every node has specific lower and upper bounds that hold with very high probability. The proofs of 
these bounds are deferred to Appendices~\ref{APP.a} and \ref{APP.b}. In Section~\ref{S3} we use the 
bounds to couple the lower and the upper external model (with rates \eqref{rlow} and \eqref{rupp}, 
respectively) to the internal model (with rates \eqref{rint}). In Section~\ref{S4} we derive the scaling
results for the external model, and combine these with the coupling to derive Theorem~\ref{thminternal} 
(as stated in Section~\ref{S1.3}).


\section{Bounds for the input and output processes}
\label{S2}

In this section we state the main results of our analysis of the input process and the output process
at a fixed node. With the help of path-large-deviation techniques, we show that with high probability 
the input process lies in a narrow tube around the deterministic path $t \mapsto (\lambda /
\mu)t$ (Proposition~\ref{inputprocess}). For simplicity, we suppress the index for the arrival rates 
$\lambda_U$ and $\lambda_V$, and consider a general rate $\lambda$. The same holds for 
$\rho = \lambda / \mu$. We study the output process only for nodes in $U$, and we give lower 
and upper bounds (Equation \eqref{outputupper} and Proposition~\ref{outputprocess}). We look at 
a single node and suppress its index, since the queues are independent of each other as 
long as the servers remain active or inactive. The proofs of the propositions below for the input 
process and the output process are given in Appendices~\ref{APP.a} and \ref{APP.b}, respectively. 

\begin{proposizione}[{\bf Tube for the input process}] 
\label{inputprocess}
For $\delta>0$ small enough and time horizon $S > 0$, let
\begin{equation}
\Gamma_{S,\delta S} = \left\{\gamma \in L_{\infty}([0,S])\colon\,
\frac{\lambda}{\mu}s  - \delta S  < \gamma(s) < \frac{\lambda}{\mu}s + \delta S
\,\,\,\forall\, s \in [0,S]\right\}.
\end{equation}
With high probability the input process lies inside $\Gamma_{S,\delta S}$ as $S\to\infty$, 
namely,
\begin{equation} 
\label{eqgamma}
\mathbb{P}\big(Q^+([0,S]) \notin \Gamma_{S, \delta S} \big) 
=  e^{-K_{\delta} S \, [1+o(1)]}, \qquad S \to \infty.
\end{equation}
\begin{equation}
K_{\delta} = (\lambda + \delta \mu) + \lambda - 2 \sqrt{\lambda(\lambda + \delta \mu)} \in (0, \infty).
\end{equation}
\end{proposizione}

\noindent
(Note that $\Gamma_{S,\delta S}$ contains negative values. This is of no concern because
the path is always non-negative.)

We want to derive lower and upper bounds for the output process for a node in $U$. An upper 
bound is trivial by definition, namely,
\begin{equation}
\label{outputupper}
Q^-(t) \leq ct, \qquad t  \geq 0.
\end{equation}
It is more delicate to compute a lower bound, for which we need some preparatory definitions. 
We first introduce an auxiliary time that will be useful in our analysis.
\begin{definizione}[{\bf Auxiliary time $T_U$}] 
\label{defTu}
Given the initial queue length $Q(0) = \gamma_U r$ for a node in $U$, define $T_U$ to be the expected 
time at which the queue length hits zero if the transition has not occurred yet. We can write 
\begin{equation} 
\label{defT_U}
T_U= T_U(r) \sim \alpha r, \qquad r \to \infty,
\end{equation}
with
\begin{equation}
\label{alphadef}
\alpha = \frac{\gamma_U}{c-\rho_U}.
\end{equation}
\hfill\qed
\end{definizione}

\noindent
\begin{remark}
{ \rm The quantity $\alpha r$ is the expected time at which the queue length hits zero when the 
node is always active. Since the total inactivity time of a node in $U$ before time $T_U$ will turn 
out to be negligible compared to $\alpha r$, we have $T_U\sim \alpha r$ as $r\to\infty$.}
\end{remark}

Next, we introduce the isolated model, an auxiliary model that will help us to derive 
a lower bound for the output process. We will see later that the internal model behaves in exactly 
the same way as the isolated model up to the pre-transition time, in particular, the pre-transition 
times in the internal and the isolated model coincide in distribution.

\begin{definizione}[{\bf Isolated model}]
In the \emph{isolated model} the activation of nodes in $U$ is not affected by the activity states of nodes 
in $V$, i.e., they behave as if they were in isolation. On the other hand, nodes in $V$ are still affected by 
nodes in $U$, i.e., they cannot activate until every node in $U$ has become inactive. Nodes in $V$ have 
zero output process.
\hfill\qed
\end{definizione}

\noindent
We study the output process for the isolated model up to time $T_U$. We will see later in 
Corollary~\ref{internalbeforeTu} that the transition time in the internal model occurs with 
high probability before $T_U$, so it is enough to look at the time interval $[0,T_U]$. In the 
rare case when the transition does not occur before $T_U$, we expect it to occur in a very 
short time after $T_U$. We are now ready to give the lower bound for the output process.

\begin{proposizione}[{\bf Lower bound for the output process in the isolated model}]
\label{outputprocess}
Consider a node in $U$. For $\delta, \epsilon, \epsilon_1, \epsilon_2 > 0$ small enough, 
the output process satisfies 
\begin{equation}
\label{eqgammatilde'}
\begin{aligned}
\mathbb{P}\big(Q^-(t) < ct-\epsilon r \,\,\, \forall t \in [0,T_U] \big) 
\leq & \,e^{-K_{\delta} \alpha r \,[1+o(1)]} + e^{-K_1 r \,[1+o(1)]}\\ 
\qquad & \,+ e^{-\big(K_2 r + K_3 \frac{r}{g_U(r)}+K_4 r \log g_U(r)\big) \,[1+o(1)]}, 
\qquad r \to \infty,
\end{aligned}
\end{equation}
with 
\begin{equation}
\begin{split}
&K_1 = \bigg(\gamma_U-\frac{2\delta \alpha}{c- \rho_U}\bigg) 
\frac{\epsilon_1 - \log(1+ \epsilon_1)}{1+\epsilon_1},\\
&K_2 = \bigg(\gamma_U-\frac{2\delta \alpha }{c- \rho_U}\bigg) (1+\epsilon_1) 
\bigg(- 1 - \log \bigg(\frac{\epsilon_2}{ \big(\gamma_U-\frac{2\delta \alpha}{c- \rho_U}\big) 
(1+\epsilon_1)}\bigg)\bigg),\\
&K_3 = \epsilon_2,\\
&K_4 = \bigg(\gamma_U-\frac{2\delta \alpha}{c- \rho_U}\bigg)(1+\epsilon_1),
\end{split}
\end{equation}
satisfying $K_1, K_2, K_3, K_4 \in (0, \infty)$.
\end{proposizione}

By combining the bounds for the input process and the output process, and picking 
$\delta = \epsilon$ and $S = r$, we obtain lower and upper bounds for the queue 
length process $Q(t)$ of a node in $U$.

\begin{corollario}[\bf Bounds for the queue length process in the isolated model]
\label{bounds}
For $\delta >0$ small enough, the following bounds hold with high probability as 
$r \to \infty$ for a node in $U$:
\begin{equation}
\begin{array}{lll}
(LB)_U\colon &Q(t) \geq Q_U^{\mathrm{LB}}(t)
 = \gamma_U r - (c -\rho_U) t - \delta r , &t \geq 0,\\[0.2cm]
(UB)_U\colon &Q(t) \leq Q_U^{\mathrm{UB}}(t)
= \gamma_U r - (c- \rho_U)t +2 \delta r , &t \in [0,T_U].
\end{array}
\end{equation}
Similarly, the following bounds hold with high probability as $r \to \infty$ for a node in $V$:
\begin{equation}
\begin{array}{lll}
(LB)_V\colon\qquad Q(t) \geq Q_V^{\mathrm{LB}}(t) 
= \gamma_V r + \rho_Vt - \delta r, &t \geq 0,\\[0.2cm]
(UB)_V\colon\qquad Q(t) \leq Q_V^{\mathrm{UB}}(t) 
= \gamma_V r + \rho_V t  + \delta r, &t \geq 0.
\end{array}
\end{equation}
\end{corollario}

\begin{proof}
The claims follow directly from Propositions~\ref{inputprocess} and \ref{outputprocess} in 
combination with \eqref{outputupper}.
\end{proof}


\section{Coupling the internal and the external model} 
\label{S3}

In Sections~\ref{sec3.1} and \ref{sec3.2} we use the bounds defined in Section~\ref{S2} to construct two 
couplings that allow us to compare the internal and the external model (Proposition~\ref{lem}, respectively, 
Proposition~\ref{proptwiddle} and Corollary~\ref{uem}). Throughout the sequel we assume that the 
deactivation rates are fixed, i.e., the deactivation Poisson clocks ring at rate 1. A node can become active only if all its neighbors are inactive. If a node is inactive, then the activation Poisson clocks ring 
at rates that vary over time in a deterministic way, or as functions of the queue lengths.

We are interested in coupling the models in the time interval $[0,T_U]$ and on the following event.

\begin{definizione}[{\bf Good behavior}] 
\label{goodbehavior}
Let $\mathcal{E}_{\delta}$ be the event that the queue length processes all have 
\textit{good behavior} in the interval $[0,T_U]$, in the sense that
\begin{equation}
\label{eventgoodbehavior}
\begin{split}
\mathcal{E}_{\delta} = & \,\big\{ Q_U^{\mathrm{LB}}(t) \leq Q_i(t) \leq Q_U^{\mathrm{UB}}(t)\,\,
\forall\, t \in [0,T_U]\,\,\forall\, i \in U \big\} \\
& \,\cup \big\{ Q_V^{\mathrm{LB}}(t) \leq Q_i(t) \leq Q_V^{\mathrm{UB}}(t)\,\,
\forall\,t \in [0,T_U] \,\,\forall\, i \in V \big\},
\end{split}
\end{equation}
i.e., the paths lie between their respectively lower and upper bounds for nodes in $U$ and $V$. 
This event depends on the perturbation parameter $\delta$.
\hfill\qed
\end{definizione}

\begin{lemma}
\label{gbprob1}
For $\delta >0$ small enough,
\begin{equation}
\lim_{r \to \infty} \mathbb{P} (\mathcal{E}_{\delta}) = 1.
\end{equation}
\end{lemma}

\begin{proof}
This follows directly from Corollary~\ref{bounds}.
\end{proof}

In what follows we couple on the event $\mathcal{E}_{\delta}$ only. The coupling can be 
extended in an arbitrary way off the event $\mathcal{E}_{\delta}$. The way this is done is 
irrelevant because of Lemma~\ref{gbprob1}. 


\subsection{Coupling the internal and the lower external model}
\label{sec3.1}

The lower external model defined in (\ref{rlow}) can also be described in the following way. 
At time $t \geq 0$ the activation rates are
\begin{equation}
r_i^{\text{low}}(t) = 
\left\{\begin{array}{ll} 
g_U(Q_U^{\mathrm{LB}}(t)), &i \in U, \\[0.2cm]
g_V(Q_V^{\mathrm{UB}}(t)), &i \in V.
\end{array}
\right. 
\end{equation}
\begin{remark}
Note that when the lower bound $Q_U^{\mathrm{LB}}(t)$ becomes negative the activation rate $g_U$ 
is zero by definition. In this way we are able extend the coupling to any time $t \geq 0$, even though 
we consider only the interval $[0,T_U]$.
\end{remark}

\begin{lemma}
\label{lemTu}
With high probability as $r \to \infty$, the transition time in the lower external model is smaller than 
$T_U$, i.e., 
\begin{equation}
\lim_{r \to \infty} \mathbb{P}_u ( \tau_v^{\mathrm{low}} \leq T_U) = 1.
\end{equation}
\end{lemma}
\begin{proof}
As we will see in Section~\ref{sec4.2}, with high probability the transition time in the external model 
is smaller than $T_U$. Since the lower external model is defined for an arbitrarily small perturbation 
$\delta >0$, we conclude by using the continuity of $g_U, g_V$.
\end{proof}

We introduce a system $\mathcal{H}^{\text{low}}$ that allows us to couple the internal model with 
the lower external model.
\begin{itemize}
\item[$\bullet$]
($\mathcal{H}^{\text{low}}$) Suppose that $h_i(t) \geq \max \{Q_U^{\mathrm{UB}}(t), 
Q_V^{\mathrm{UB}}(t) \}$ for all $i \in U \cup V$ and all $t \in [0, T_U]$. Consider a 
system $\mathcal{H}^{\text{low}}$ where clocks are associated with each node in the 
following way:
\begin{itemize}
\item 
A Poisson deactivation clock ticks at rate 1. Both the nodes in the lower external model and 
in the internal model are governed by this clock:
\begin{itemize}
\item[--] 
if both nodes are active, then they become inactive together; 
\item[--]
if only one node is active, then it becomes inactive; 
\item[--]
if both nodes are inactive, then nothing happens.
\end{itemize}
\item 
A Poisson activation clock ticks at rate $g_U(h_i(t))$ at time $t$ for a node $i \in U$. 
Both the nodes in the lower external model and in the internal model are governed by 
this clock:
\begin{itemize}
\item[--]
if both nodes are active, or both are inactive but have active neighbors, then nothing happens; 
\item[--]
if the node in the internal model is active and the node in the lower external model is not, 
then the latter node becomes active (if it can) with probability 
\begin{equation}
\frac{r_i^{\text{low}}(t)}{g_U(h_i(t))};
\end{equation}
\item[--]
if both nodes are inactive but can be activated, then this happens with probabilities 
\begin{equation}
\begin{aligned}
&\frac{r_i^{\text{low}}(t)}{g_U(h_i(t))} \quad \text{ for the lower external model},\\
&\frac{r_i^{\text{int}}(t)}{g_U(h_i(t))} \quad \text{ for the internal model},
\end{aligned}
\end{equation}
where 
\begin{equation}
\label{bd1}
\frac{r_i^{\text{low}}(t)}{g_U(h_i(t))} \leq \frac{r_i^{\text{int}}(t)}{g_U(h_i(t))},
\end{equation}
in such a way that if the node in the lower external model activates, then it also activates 
in the internal model.
\end{itemize}
\item 
A Poisson activation clock ticks at rate $g_V(h_i(t))$ at time $t$ for a node $i \in V$. The 
same happens as for the nodes in $U$, but the activation probabilities are 
\begin{equation}
\begin{aligned}
&\frac{r_i^{\text{low}}(t)}{g_V(h_i(t))} \quad \text{ for the lower external model},\\ 
&\frac{r_i^{\text{int}}(t)}{g_V(h_i(t))} \quad \text{ for the internal model},
\end{aligned}
\end{equation} 
where 
\begin{equation}
\label{bd2}
\frac{r_i^{\text{low}}(t)}{g_U(h_i(t))} \geq \frac{r_i^{\text{int}}(t)}{g_U(h_i(t))},
\end{equation}
in such a way that if the node in the internal model activates, then it also activates in the lower 
external model.
\end{itemize}
\end{itemize}

With the constructions above, we are now able to compare the transition times of the two 
models. 

\begin{proposizione}[{\bf Comparison between the internal and the lower external model}] 
\label{lem}
$\mbox{}$
\begin{itemize}
\item[(i)] 
Under the coupling $\mathcal{H}^\mathrm{low}$, the joint activity processes in the internal and 
in the lower external model are ordered for all $t \in [0, T_U]$, i.e.,
\begin{equation}
\begin{aligned}
&X_i^{\mathrm{low}}(t) \leq X_i^{\mathrm{int}}(t), &i \in U, \\
&X_i^{\mathrm{int}}(t) \leq X_i^{\mathrm{low}}(t), &i \in V.
\end{aligned}
\end{equation}
\item[(ii)] 
With high probability as $r \to \infty$, the transition time $\tau_v^{\mathrm{int}}$ in the internal 
model is at least as large as the transition time $\tau_v^{\mathrm{low}}$ in the lower external model, 
i.e.,
\begin{equation}
\label{eq.lemaltalt}
\lim_{r \to \infty} \mathbb{\hat{P}}_u (\tau_v^{\mathrm{low}} \leq \tau_v^{\mathrm{int}}) =1,
\end{equation}
where $\mathbb{\hat{P}}_u$ is the joint law induced by the coupling with starting configuration $u$.
\end{itemize}
\end{proposizione}

\begin{proof}
\begin{itemize}
\item[(i)] 
For each node $i \in U$ and for all $t \in [0, T_U]$, $Q_i^{\mathrm{LB}}(t) \leq Q_i(t)$ and 
$g_U(Q_i^{\mathrm{LB}}(t)) \leq g_U(Q_i(t))$ by the monotonicity of the function $g_U$. 
On the other hand, for each node $i \in V$, $Q_i(t) \leq Q_i^{\mathrm{UB}}(t)$ and 
$g_V(Q_i(t)) \leq g_V(Q_i^{\mathrm{UB}}(t))$ by the monotonicity of the function $g_V$. 
Under the system $\mathcal{H}^{\text{low}}$, at any moment the random variable describing 
the state of a node $ i \in U$ in the lower external model is dominated by the one in the 
internal model, i.e., by \eqref{bd1} for all $t \in [0, T_U]$,
\begin{equation}
X_i^{\text{low}}(t) \leq X_i^{\text{int}}(t).
\end{equation}
On the other hand, the random variable describing the state of a node $j \in V$ in the lower 
external model dominates the one in the internal model, i.e., by \eqref{bd2} for all $t \in [0, T_U]$,
\begin{equation}
X_i^{\text{int}}(t) \leq X_i^{\text{low}}(t).
\end{equation}
Hence the joint activity processes in the two models are ordered.
\item[(ii)]
By construction of the coupling and the ordering above, on the event $\mathcal{E}_{\delta}$ 
the nodes in $U$ in the lower external model turn off earlier than in the internal model, and also 
nodes in $V$ turn on earlier in the lower external model. Hence the transition occurs earlier
in the lower external model.

Note that we are able to compare the transition times only when $\tau_v^{\mathrm{low}} \leq T_U$, 
so we look at the coupling also on the event $\{\tau_v^{\mathrm{low}} \leq T_U\}$, which has high 
probability as $r \to \infty$ (Lemma~\ref{lemTu}). On this event we have $\tau_v^{\mathrm{low}} 
\leq \tau_v^{\mathrm{int}}$. Therefore 
\begin{equation}
1 =  \lim_{r \to \infty} \mathbb{\hat{P}}_u (\mathcal{E}_{\delta}, \tau_v^{\mathrm{low}} 
\leq T_U, \tau_v^{\mathrm{low}} \leq \tau_v^{\mathrm{int}}) 
= \lim_{r \to \infty} \mathbb{\hat{P}}_u ( \tau_v^{\mathrm{low}} \leq \tau_v^{\mathrm{int}}).
\end{equation}
\end{itemize}
\end{proof}


\subsection{Coupling the isolated and the upper external model}
\label{sec3.2}

The upper external model defined in (\ref{rupp}) can also be described in the following way. At time 
$t \in [0,T_U]$ the activation rates are
\begin{equation}
r_i^{\text{upp}}(t) = 
\left\{\begin{array}{ll}
g_U(Q_U^{\mathrm{UB}}(t)), &i \in U, \\[0.2cm] 
g_V(Q_V^{\mathrm{LB}}(t)), &i \in V. 
\end{array}
\right. 
\end{equation}

\begin{lemma}
\label{uemTu}
With high probability as $r \to \infty$, the transition time in the upper external model is smaller than 
$T_U$, i.e., 
\begin{equation}
\lim_{r \to \infty} \mathbb{P}_u (\tau_v^{\mathrm{upp}} \leq T_U) = 1.
\end{equation}
This statement is to be read as follows. Let $\delta$ be the perturbation parameter in the upper external 
model appearing in \eqref{rupp}. Then for every $\delta>0$ there exists a $\delta'(\delta)>0$, satisfying 
$\lim_{\delta \downarrow 0} \delta'(\delta) = 0$, such that $\lim_{r \to \infty} \mathbb{P}_u
(\tau_v^{\mathrm{upp}} \leq [1+\delta'(\delta)]T_U) = 1$.   
\end{lemma}

\begin{proof}
Analogous to the proof of Lemma~\ref{lemTu}. 
\end{proof}

We introduce a system $\mathcal{H}^{\text{upp}}$ that allows us to couple the isolated model 
with the upper external model up to time $\bar{\tau}_v^{\mathrm{iso}}$.
\begin{itemize}
\item[$\bullet$]
($\mathcal{H}^{\text{upp}}$) Suppose that $h_i(t) \geq \max \{Q_U^{\mathrm{UB}}(t),Q_V^{\mathrm{UB}}(t)\}$ 
for all $i \in U \cup V$ and all $t \in [0, \bar{\tau}_v^{\mathrm{iso}}]$. Couple the processes in the same way 
as for $\mathcal{H}^{\text{low}}$, but with different activation probabilities. The probabilities for the isolated 
model and for the upper external model are such that
\begin{equation}
\begin{aligned}
&\frac{r_i^{\text{iso}}(t)}{g_U(h_i(t))} \leq \frac{r_i^{\text{upp}}(t)}{g_U(h_i(t))},
&i \in U,\\
&\frac{r_i^{\text{upp}}(t)}{g_V(h_i(t))} \leq \frac{r_i^{\text{iso}}(t)}{g_V(h_i(t))},
&i \in V,
\end{aligned}
\end{equation} 
where for $t \in [0,\bar{\tau}_v^{\mathrm{iso}}]$
\begin{equation}
r_i^{\text{iso}}(t) = 
\left\{\begin{array}{ll}
g_U(Q_i(t)), &i \in U, \\[0.2cm] 
g_V(Q_i(t)), &i \in V. 
\end{array}
\right. 
\end{equation}
\end{itemize}

Note that when $\bar{\tau}_v^{\mathrm{iso}} \leq T_U$, the isolated model behaves exactly as the internal 
model in the interval $[0, \bar{\tau}_v^{\mathrm{iso}}]$, as shown in Appendix~\ref{APP.b2}. Moreover, the 
coupling is defined only when $\bar{\tau}_v^{\mathrm{iso}} \leq T_U$. We look then at the coupling also 
on the event $\{\tau_v^{\mathrm{upp}} \leq T_U\}$, which has high probability as $r \to \infty$ 
(Lemma~\ref{uemTu}). In the following proposition we see how this ensures that the coupling is well 
defined, and we compare the pre-transition times of the two models.

\begin{proposizione}[{\bf Comparison between the isolated and the upper external model}] 
\label{proptwiddle}
$\mbox{}$
\begin{itemize}
\item[(i)] 
Under the coupling $\mathcal{H}^{\mathrm{upp}}$, the joint activity processes in the isolated
model and in the upper external model are ordered up to time $\bar{\tau}_v^{\mathrm{iso}}$, 
i.e., for all $t \in [0,\bar{\tau}_v^{\mathrm{iso}}]$,
\begin{equation}
\begin{aligned}
&X_i^{\mathrm{iso}}(t) \leq X_i^{\mathrm{upp}}(t), &i \in U,\\
&X_i^{\mathrm{upp}}(t) \leq X_i^{\mathrm{iso}}(t), &i \in V.
\end{aligned}
\end{equation}
\item[(ii)] 
With high probability as $r \to \infty$, the pre-transition time $\bar{\tau}_v^{\mathrm{upp}}$ in 
the upper external model is at least as large as the pre-transition time $\bar{\tau}_v^{\mathrm{iso}}$ in the 
isolated model, i.e.,
\begin{equation}
\label{eq.lem}
\lim_{r \to \infty} \mathbb{\hat{P}}_u (\bar{\tau}_v^{\mathrm{iso}} \leq \bar{\tau}_v^{\mathrm{upp}}) = 1,
\end{equation}
where $\mathbb{\hat{P}}_u$ is the joint law induced by the coupling with starting configuration $u$.
\end{itemize}
\end{proposizione}

\begin{proof}
\begin{itemize}
\item[(i)]
The proof is analogous to that of Proposition~\ref{lem}, but this time we use the system 
$\mathcal{H}^{\text{upp}}$ up to time $\bar{\tau}_v^{\mathrm{iso}}$ and all the inequalities 
are reversed. 
\item[(ii)]
By construction of the coupling and the ordering above, on the event $\mathcal{E}_{\delta}
\cap \{\tau_v^{\mathrm{upp}} \leq T_U\}$ the nodes in $U$ in the isolated model turn off 
earlier than in the upper external model, and also the first node in $V$ turns on earlier in 
the isolated model. Hence the pre-transition occurs earlier in the isolated model, and we 
have $\bar{\tau}_v^{\mathrm{iso}} \leq \bar{\tau}_v^{\mathrm{upp}} \leq \tau_v^{\mathrm{upp}} 
\leq T_U$. Therefore the coupling is well defined and 
\begin{equation}
1 =  \lim_{r \to \infty} \mathbb{\hat{P}}_u (\mathcal{E}_{\delta,T_U}, \tau_v^{\mathrm{upp}} 
\leq T_U,  \bar{\tau}_v^{\mathrm{iso}} \leq \bar{\tau}_v^{\mathrm{upp}}) 
=  \lim_{r \to \infty} \mathbb{\hat{P}}_u ( \bar{\tau}_v^{\mathrm{iso}} 
\leq \bar{\tau}_v^{\mathrm{upp}}).
\end{equation}
\end{itemize}
\end{proof}

\begin{corollario}
\label{uem}
With high probability as $r \to \infty$, the transition time $\tau_v^{\mathrm{upp}}$ in the upper 
external model is at least as large as the pre-transition time $\bar{\tau}_v^{\mathrm{int}}$ in the 
internal model, i.e.,
\begin{equation}
\label{eq.lemalt}
\lim_{r \to \infty} \mathbb{\hat{P}}_u (\bar{\tau}_v^{\mathrm{int}} \leq \tau_v^{\mathrm{upp}}) = 1.
\end{equation}
\end{corollario}

\begin{proof}
Since $\lim_{r \to \infty} \mathbb{P}(\bar{\tau}_v^{\mathrm{iso}} \leq T_U) = 1$, we have, as shown in 
Proposition~\ref{prop:coincide} in Appendix~\ref{APP.b2}, that the pre-transition times in the isolated 
model and in the internal model coincide. Hence
\begin{equation}
1= \lim_{r \to \infty} \mathbb{\hat{P}}_u (\bar{\tau}_v^{\mathrm{iso}} \leq \bar{\tau}_v^{\mathrm{upp}}) 
=\lim_{r \to \infty} \mathbb{\hat{P}}_u (\bar{\tau}_v^{\mathrm{int}} \leq \bar{\tau}_v^{\mathrm{upp}}) 
\leq \lim_{r \to \infty} \mathbb{\hat{P}}_u (\bar{\tau}_v^{\mathrm{int}} \leq \tau_v^{\mathrm{upp}}).
\end{equation}
\end{proof}


\section{Transition times}
\label{S4}

The goal of this section is to identify the asymptotic behavior of the transition time in the 
internal model. In Sections~\ref{sec4.1} and \ref{sec4.2} we look at the external model and 
prove Theorems~\ref{thmMc} and Theorem~\ref{thmexternal}, respectively. In Section~\ref{sec4.3}
we show that the difference between the transition time and the pre-transition time is negligible
for all the models  considered in the paper. In Section~\ref{sec4.4} we put these results together 
to prove Theorem~\ref{thminternal}.


\subsection{Critical time scale in the external model}
\label{sec4.1}

In this section we prove Theorem~\ref{thmMc}. Below, $a(r) \sim b(r)$ means that 
$\lim_{r \to \infty} a(r)/b(r) =1$, while $a(r) \asymp b(r)$ means that $0 < \liminf_{r\to\infty} 
a(r)/b(r) \leq \limsup_{r\to\infty} a(r)/b(r) < \infty$. 

\begin{proof}
In order to compute the critical time scale $M_c$, we must solve the equation $M \nu(M) = 1$ 
in (\ref{Mc}). We know from (\ref{nus}) and (\ref{meantaucomplete}) that 
\begin{equation}
\nu(s) \sim |U| r_U(s)^{1- |U|}, \qquad r \to \infty.
\end{equation}
We want to identify how the transition time is related to the choice of the activation function $g_U(x) 
= G x^{\beta}$ with $\beta,G \in (0, \infty)$. Consider the time scale $M_c = F_cr^\gamma$, where 
$\gamma \in (0,1]$ and $F_c \in (0,\infty)$. For $r \to \infty$ we have
\begin{equation} 
\label{eqMcf}
\begin{split}
1 = r^0 
&= M_c \nu(M_c) = F_c r^\gamma\, \nu(F_c r^\gamma) \sim F_c r^\gamma |U| r_U(F_c r^\gamma)^{-(|U|-1)} \\
&= F_c r^\gamma |U| g_U\big(\gamma_U r - (c- \rho_U) F_c r^\gamma \big)^{-(|U|-1)} \\
&= F_c r^\gamma |U| G^{-(|U|-1)} \big(\gamma_U r - (c- \rho_U) F_c r^\gamma \big)^{- \beta (|U|-1)}.
\end{split}
\end{equation}
Recall from \eqref{alphadef} that $\alpha = \frac{\gamma_U}{c- \rho_U}$. We consider three cases: 
\begin{itemize}
\item {\bf Case $\gamma \in (0,1)$ and $F_c \in (0,\infty)$.}
For $r \to \infty$ the criterion in \eqref{eqMcf} reads
\begin{equation} 
\label{eqMc}
1 = r^0 \sim F_cr^{\gamma} |U| G^{-(|U|-1)} (\gamma_U r )^{- \beta (|U|-1)}.
\end{equation}
In order for the exponents of 
$r$ to match, we need 
\begin{equation}
\label{beta*}
\beta = \frac{\gamma}{|U|-1}.
\end{equation}
Inserting (\ref{beta*}) into (\ref{eqMc}), we get
\begin{equation}
F_c |U| G^{-(|U|-1)} \gamma_U ^{- \beta (|U|-1)}= 1,
\end{equation}
which gives 
\begin{equation}
\label{F_c, case 1}
F_c = \frac{\gamma_U^{\beta (|U|-1)} G^{(|U|-1)}}{|U|}.
\end{equation}
Hence 
\begin{equation}
M_c=\frac{(G \gamma_U^{\beta})^{|U|-1}}{|U|} r^{\beta (|U|-1)}.
\end{equation}

\item {\bf Case $\gamma =1$ and $F_c \in (0,\alpha)$.}
For $r \to \infty$ the criterion in \eqref{eqMcf} reads
\begin{equation}
\label{rel1}
1 =  r^0 \sim  F_c |U|  G^{-(|U|-1)}(\gamma_U -(c-\rho_U)F_c)^{-\beta(|U|-1)} r^{1-\beta(|U|-1)}.
\end{equation}
In order for the exponents of $r$ to match, we need 
\begin{equation}
\label{rel2}
\beta = \frac{1}{|U|-1}.
\end{equation}
Inserting \eqref{rel2} into \eqref{rel1}, we get
\begin{equation}
\frac{F_c |U| G^{-(|U|-1)}}{\gamma_U - (c- \rho_U)F_c} = 1,
\end{equation}
which gives 
\begin{equation}
\label{F_c, case 2}
F_c = \frac{\gamma_U}{|U|  G^{-(|U|-1)}+ (c- \rho_U)}.
\end{equation}
Hence
\begin{equation}
M_c = \frac{\gamma_U}{|U|  G^{-(|U|-1)}+ (c- \rho_U)}r.
\end{equation} 
Recall from \eqref{defT_U} that $T_U \sim \alpha r$ is the expected time at which the queue length at a node 
in $U$ hits zero. We will see in Section~\ref{sec4.2} that the transition in the external 
model typically occurs before the queues are empty.

\item {\bf Case $\gamma =1$ and $F_c =\alpha- Dr^{-\delta}$, $\delta \in (0,1)$.} 
For $r \to \infty$ the criterion in \eqref{eqMcf} reads
\begin{equation}
\label{rel1*}
1 =  r^0 \sim  \alpha r |U|  G^{-(|U|-1)} \big((c- \rho_U)D r^{1-\delta}\big)^{-\beta (|U|-1)}.
\end{equation}
In order for the exponents of $r$ to match, we need 
\begin{equation}
\label{rel2*}
\beta = \frac{1}{(1-\delta)(|U|-1)}.
\end{equation}
Inserting \eqref{rel2*} into \eqref{rel1*}, we get
\begin{equation}
\alpha|U|G^{-(|U|-1)}((c-\rho_U)D)^{-\beta(|U|-1)} = 1,
\end{equation}
which gives
\begin{equation}
D = \frac{(\alpha|U|G^{-(|U|-1)})^{1/\beta(|U|-1)}}{c-\rho_U}.
\end{equation}
Hence 
\begin{equation}
M_c = \alpha r - \frac{(\alpha|U|G^{-(|U|-1)})^{1/\beta(|U|-1)}}{c-\rho_U} r^{1/\beta(|U|-1)},
\end{equation} 
and so the crossover takes place in a window of size $O(r^{1/\beta(|U|-1)}) = o(r)$ around $\alpha r$. 
Note that this window gets narrower as $\beta$ increases, i.e., as the activation rate for nodes 
in $U$ increases.
\end{itemize}
\end{proof}

\begin{remark}[{\bf Scaling revisited}]
{\rm As noted in Remark~\ref{regimes}, the formula in \cite[Theorem 1.3]{BdHNT17} is different from \eqref{siamakpaper2}. However, as shown in \cite[Example 1.4]{BdHNT17}, for complete bipartite graphs both 
formulas lead to the same asymptotics in the subcritical regime, and also in the critical and supercritical 
regime provided $x \leq \frac{\gamma_U}{c-\rho_U} \frac{r}{M_c} = \frac{\gamma_U}{c-\rho_U}
\frac{1}{F_c}$ (see \cite[Eq.(1.22)]{BdHNT17}). The latter becomes $x \leq \frac{1}{C}$ in the critical 
regime and $x \leq 1$ in the supercritical regime, which precisely matches Fig.~\ref{fig:dichotomy}. 
The above computations constitute a refinement of the computations in \cite{BdHNT17}.}
\end{remark}

\begin{remark}[{\bf Modulation with slowly varying functions}] 
\label{remarksvf}
{\rm We may consider the more general case where the activation function is $g_U(x) = x^{\beta} 
\hat{\mathcal{L}}(x)$ with $\beta \in (0,\infty)$ and $\hat{\mathcal{L}}(x)$ a slowly varying function
(i.e., $\lim_{x \to \infty} \hat{\mathcal{L}}(ax)/\hat{\mathcal{L}}(x) = 1$ for all $a >0$). Let 
$M_c = r^{\gamma} \mathcal{L}(r)$ with $\gamma \in (0,1)$ and $\mathcal{L}(r)$ a slowly 
varying function. as $r\to\infty$, we have
\begin{equation} \label{eqMcSV}
\begin{split}
1 = r^0 &\sim r^{\gamma}\mathcal{L}(r) |U| \big(\gamma_U r - (c- \rho_U) 
r^{\gamma}\mathcal{L}(r)\big)^{- \beta (|U|-1)} \hat{\mathcal{L}}
(\gamma_U r - (c- \rho_U) r^{\gamma}\mathcal{L}(r))^{-(|U|-1)} \\
&\sim  r^{\gamma}\mathcal{L}(r) |U| (\gamma_U r )^{- \beta (|U|-1)} 
\hat{\mathcal{L}}(\gamma_U r )^{-(|U|-1)}.
\end{split}
\end{equation}
In order for the exponents of $r$ to match, we again need
\begin{equation}
\label{rel3}
\beta = \frac{\gamma}{|U|-1}. 
\end{equation}
We get
\begin{equation}
\mathcal{L}(r) = \frac{\gamma_U^{\beta(|U|-1)}}{|U|} \hat{\mathcal{L}}(\gamma_U r)^{|U|-1}.
\end{equation}
Hence 
\begin{equation}
M_c = \frac{\gamma_U^{\beta(|U|-1)}}{|U|} r^{\beta (|U|-1)} \hat{\mathcal{L}}(\gamma_U r)^{|U|-1}.
\end{equation}
We can even include the case $\beta=0$, and get that if $g_U(x) =  \hat{\mathcal{L}}(x)$ with 
$\lim_{x\to\infty} \hat{\mathcal{L}}(x)=\infty$, then 
\begin{equation}
M_c = \frac{1}{|U|}  \hat{\mathcal{L}}(\gamma_U r)^{|U|-1}.
\end{equation}
}
\end{remark}

\subsection{Transition time in the external model} 
\label{sec4.2}

In this section we prove Theorem~\ref{thmexternal}. We already know that the transition occurs on 
the critical time scale $M_c$ computed in Section~\ref{sec4.2}.

\begin{proof}
Knowing the critical time scale $M_c$, we can compute the mean transition time from 
(\ref{siamakpaper2}):
\begin{equation}
\begin{split}
\mathbb{E}_{u} [ \tau_v^{\mathrm{ext}}] = & \, \int_0^{\infty} 
\mathbb{P}_u(\tau_v^{\mathrm{ext}} > x)\, dx 
= M_c \int_0^{\infty} \mathbb{P}_u\bigg(\frac{\tau_v^{\mathrm{ext}}}{M_c} > x\bigg)\, dx \\
\sim & \, M_c \int_0^{\infty} e^{-\int_0^x M_c \nu(M_c s)\,ds}\, dx = M_c \int_0^{\infty} 
e^{-\int_0^x \frac{M_c \nu(M_c s)}{M_c \nu(M_c)} ds}\, dx \\
= & \,  M_c \int_0^{\infty} e^{-\int_0^x \big(\frac{\gamma_U r 
- (c-\rho_U) M_c s}{\gamma_U r - (c- \rho_U) M_c}\big)^{-\beta(|U|-1)}\, ds}\, dx, \qquad r \to \infty,
\end{split}
\end{equation}
where the choice of $\beta$ is important. 

\begin{itemize}
\item{\bf Case $\beta \in (0,\frac{1}{|U|-1})$, $M_c= F_c r^{\gamma}$, $\gamma \in (0,1)$.}\\
We have 
\begin{equation}
\left(\frac{\gamma_U r - (c-\rho_U) M_c s}{\gamma_U r 
- (c-\rho_U) M_c}\right)^{-\beta(|U|-1)} \to 1, \qquad r \to \infty.
\end{equation} 
Hence 
\begin{equation}
\mathbb{E}_{u} [ \tau_v^{\mathrm{ext}}] \sim M_c \int_0^{\infty} e^{-\int_0^x  ds}\, dx 
= M_c \int_0^{\infty} e^{-x}\, dx 
= M_c, \qquad r \to \infty.
\end{equation}
The law of $\tau_v^{\mathrm{ext}}$ is exponential, i.e.,
\begin{equation}
\lim_{r \to \infty} \mathbb{P}_u\bigg(\frac{\tau_v^{\mathrm{ext}}}
{\mathbb{E}_u[\tau_v^{\mathrm{ext}}]} > x\bigg) = e^{-x}, \qquad x \in [0, \infty).
\end{equation}

\item{\bf Case $\beta = \frac{1}{|U|-1}$, $M_c = F_c r$, $F_c \in (0,\alpha)$.}
We have
\begin{equation}
\begin{split}
\bigg(\frac{\gamma_U r - (c- \rho_U) F_c r s}{\gamma_U r - (c-\rho_U) F_c r}\bigg)^{-\beta(|U|-1)} 
= & \,\frac{\gamma_U - (c- \rho_U) F_c }{\gamma_U - (c-\rho_U) F_c s}
= \frac{1 - \frac{c- \rho_U}{\gamma_U} F_c }{1 - \frac{c- \rho_U}{\gamma_U} F_c s}\\
= & \, \frac{1 - \frac{F_c}{\alpha} }{1 - \frac{F_c}{\alpha} s} = \frac{1 - C }{1 - C s},
\end{split}
\end{equation}
with $C = F_c /\alpha$. Hence 
\begin{equation}
\begin{aligned}
\mathbb{E}_{u} [ \tau_v^{\mathrm{ext}}] 
&\sim M_c \int_0^{1/C} e^{-\int_0^x \frac{1-C}{1-Cs} ds}\, dx 
= M_c \int_0^{1/C} e^{-\log(1-Cx)^{-\frac{1-C}{C}}}\, dx \\
&= M_c \int_0^{1/C} (1-Cx)^{\frac{1-C}{C}}\, dx
= M_c \bigg[ (1-Cx)^{1 +\frac{1-C}{C}} \frac{1}{(1 +\frac{1-C}{C})(-C)}\bigg]_0^{1/C} \\
&= M_c \bigg[ -(1-Cx)^{\frac{1}{C}} \bigg]_0^{1/C}  =  M_c, \qquad r \to \infty.
\end{aligned}
\end{equation}
Here, the integral must be truncated at $x=1/C$ because for larger $x$ the integrand 
becomes negative. Indeed, note that when $x = 1/C = \alpha/F_c$, which corresponds 
to time $T_U = \alpha r$, we have 
\begin{equation}
\begin{split}
\lim_{r \to \infty} \mathbb{P}_u\big(\tau_v^{\mathrm{ext}} > T_U \big) 
= & \lim_{r \to \infty} \mathbb{P}_u\bigg(\tau_v^{\mathrm{ext}} > \frac{\alpha}{F_c} F_c r \bigg) 
= \lim_{r \to \infty} \mathbb{P}_u\bigg(\frac{\tau_v^{\mathrm{ext}}}
{M_c} > \frac{\alpha}{F_c} \bigg) \\
 = & \left(1- C \frac{\alpha}{F_c}\right)^{\frac{1-C}{C}} = 0,
\end{split}
\end{equation}
because $C = F_c/\alpha$.  This means that, with high probability as $r \to \infty$, 
the transition occurs before time $T_U$. The law of $\tau_v^{\mathrm{ext}}$ is truncated polynomial: 
\begin{equation}
\label{pollimlaw}
\lim_{r \to \infty} \mathbb{P}_u\bigg(\frac{\tau_v^{\mathrm{ext}}}
{\mathbb{E}_u[\tau_v^{\mathrm{ext}}]} > x\bigg) = 
\left\{\begin{array}{ll}
(1- Cx)^{\frac{1-C}{C}}, &x  \in [0,\frac{1}{C}), \\[0.2cm]
0, &x \in [\frac{1}{C},\infty).
\end{array}
\right.
\end{equation}

\item{\bf Case $\beta \in (\frac{1}{|U|-1},\infty)$, $M_c = \alpha r$.}
This case corresponds to the limit $C \uparrow 1$ of the previous case. In this limit, \eqref{pollimlaw} 
becomes
\begin{equation}
\lim_{r \to \infty} \mathbb{P}_u\bigg(\frac{\tau_v^{\mathrm{ext}}}
{\mathbb{E}_u[\tau_v^{\mathrm{ext}}]} > x\bigg) = 
\left\{\begin{array}{ll}
1, &x  \in [0,1), \\[0.2cm]
0, &x \in [1,\infty).
\end{array}
\right.
\end{equation}
\end{itemize}
\end{proof}


\subsection{Negligible gap in the internal model}
\label{sec4.3}

In this section we focus on the internal model and estimate the length of the interval 
$[\bar{\tau}_v^{\text{int}}, \tau_v^{\text{int}}]$, which turns out to be very small with 
high probability. This implies that the transition time has the same asymptotic behavior 
as the pre-transition time. 

We know that the queue at a node $i \in V$ is of order $r$ at time $\bar{\tau}_v^{\mathrm{int}}$, 
i.e., $Q_i(\bar{\tau}_v^{\mathrm{int}}) \asymp r$, since it starts at $\gamma_V r$, with 
$\gamma_V > 0$, and only the input process is present until this time. Hence all the activation 
Poisson clocks at nodes in $V$ tick at a very aggressive rate. The idea is that within the 
activation period (which has an exponential distribution with mean 1) of the first node activating 
in $V$, all the other nodes in $V$ become active because they are not ``blocked" by any node
in $U$. Consequently, the system quickly reaches the configuration $v$. 

\begin{teorema}[\bf Negligible gap]
\label{cutoff}
In the internal model 
\begin{equation}
\lim_{r \to \infty} \mathbb{P}_u \bigg( \tau_v^{\mathrm{int}} - \bar{\tau}_v^{\mathrm{int}} 
= o\bigg(\frac{1}{g_V(r)}\bigg) \bigg) = 1.
\end{equation}
\end{teorema}

\begin{proof}
Starting from $\bar{\tau}_v^{\text{int}}$, a node $x \in V$ remains inactive for an exponential 
period with mean $1/r_x^{\mathrm{int}}(\bar{\tau}_v) = 1/g_V(Q(\bar{\tau}_v)) 
\asymp 1/g_V(r)$. Denote by $W_x$ the length of an inactivity period for a node $x \in V$. 
We have i.i.d.\ inactivity periods $W_x \simeq \text{Exp}(g_V(Q(\bar{\tau}_v)))$ for all $x \in V 
\setminus \{x_1\}$, where $x_1$ is the first node activating in $V$. We label the remaining 
nodes $x_2, \dots, x_{|V|}$ in an arbitrary way. We also have i.i.d.\ activity periods $Z_x 
\simeq \text{Exp}(1)$ for all $x \in V$. 

Consider a time $t_1 = o(1/g_V(r))$. With high probability all the nodes in 
$V$ activate before time $t_1$, i.e.,
\begin{equation}
\mathbb{P}\big(  W_{x_i} < t_1, \, \forall\, i = 2, \dots, |V| \big) 
= \mathbb{P}\big( W_{x_2} < t_1 \big)^{|V|-1} 
= \big(1- e^{-g_V(Q(\bar{\tau}_v)) t_1}\big)^{|V|-1} \xrightarrow[r \to \infty] \, 1.
\end{equation}
Moreover, with high probability, once activated, all nodes in $V$ stay active for a period of 
length at least $t_2 \asymp 1/g_V(r)> t_1$, i.e.,
\begin{equation}
\mathbb{P}\big(  Z_{x_i} > t_2 \,\,\, \forall\, i = 1, \dots, |V| \big) 
= \mathbb{P}\big( Z_{x_1} > t_2 \big)^{|V|} = ( e^{- t_2})^{|V|} \xrightarrow[r \to \infty] \, 1.
\end{equation}
In conclusion, as $r \to \infty$, with high probability every node in $V$ activates before time $t_1$ and 
remains active for at least a time $t_2>t_1$. This ensures that the transition occurs before time $t_2$. In 
particular, it occurs when the last node in $V$ activates (which occurs even before time $t_1$), so 
that $\tau_v^{\text{int}} - \bar{\tau}_v^{\text{int}} = o(1/g_V(r))$.
\end{proof}

Note that this argument extends to any ``external''  model with activation rates that tend to infinity with $r$, 
in particular, to all the models considered in this paper. The transition always happens quickly after the 
pre-transition, due to the high level of aggressiveness of nodes in $V$.

\begin{corollario}
\label{internalbeforeTu}
With high probability as $r \to \infty$, the transition time in the internal model is smaller than $T_U$, i.e.,
\begin{equation}
\lim_{r \to \infty} \mathbb{P}_u ( \tau_v^{\mathrm{int}} \leq T_U) = 1.
\end{equation}
\end{corollario}
\begin{proof}
This follows immediately from Lemma~\ref{uemTu}, Corollary~\ref{uem} and Theorem~\ref{cutoff}.

\end{proof}


\subsection{Transition time in the internal model}
\label{sec4.4}

In this section we prove Theorem~\ref{thminternal}. First we derive the sandwich of the transition times 
in the lower external, the internal and the upper external model. After that we identify the asymptotics of 
the transition time for the internal model by using the results for the external models.

\begin{proof}
Using Proposition~\ref{lem}, Corollary~\ref{uem} and Theorem~\ref{cutoff}, we have that there exists 
a coupling such that
\begin{equation}
\label{sandwichtransition}
\begin{split}
1= &\, \lim_{r \to \infty} \mathbb{\hat{P}}_u \bigg(\tau_v^{\mathrm{low}} 
\leq \tau_v^{\mathrm{int}}, \tau_v^{\mathrm{int}} 
= \bar{\tau}_v^{\mathrm{int}} + o\bigg(\frac{1}{g_V(r)}\bigg), \bar{\tau}_v^{\mathrm{int}} 
\leq \tau_v^{\mathrm{upp}} \bigg) \\
= & \, \lim_{r \to \infty} \mathbb{\hat{P}}_u \bigg(\tau_v^{\mathrm{low}} 
\leq \tau_v^{\mathrm{int}} \leq \tau_v^{\mathrm{upp}} + o\bigg(\frac{1}{g_V(r)} \bigg) \bigg) \\
= &\, \lim_{r \to \infty} \mathbb{\hat{P}}_u \big(\tau_v^{\mathrm{low}} 
\leq \tau_v^{\mathrm{int}} \leq \tau_v^{\mathrm{upp}} \big),
\end{split}
\end{equation}
where $\mathbb{\hat{P}}_u$ is the joint law of the three models on the same probability 
space all three starting from configuration $u$.

By Theorem~\ref{thmexternal}, we know the law of the transition time in the external model. By 
construction, we have $\mathbb{E}_{u}[\tau_v^{\mathrm{low}}] \leq \mathbb{E}_{u}[\tau_v^{\mathrm{ext}}] 
\leq \mathbb{E}_{u} [ \tau_v^{\mathrm{upp}}]$. When considering the lower and the upper external 
model, the transition time asymptotics are controlled by the prefactors $F_{c,\delta}^{\text{low}}$ 
and $F_{c,\delta}^{\text{upp}}$, respectively, which are perturbations of the prefactor $F_c$ due to 
the perturbations of the activation rates. In particular, we know from \eqref{effe} that $\lim_{\delta 
\downarrow 0} F_{c,\delta}^{\text{low}} = \lim_{\delta \downarrow 0} F_{c,\delta}^{\text{upp}} = F_c$. 
Hence, for all $\epsilon > 0$, 
\begin{equation}
\mathbb{E}_{u} [ \tau_v^{\mathrm{int}}] = (F_c \pm \epsilon) r^{\beta (|U|-1)} \, [1+o(1)], 
\qquad r \to \infty,
\end{equation}
and since $\epsilon$ can be taken arbitrarily small, it may be absorbed into the $o(1)$-term, as
\begin{equation}
\mathbb{E}_{u} [ \tau_v^{\mathrm{int}}] = F_cr^{\beta (|U|-1)} \, [1+o(1)], \qquad r \to \infty.
\end{equation}
The same kind of argument applies to the law of the transition time, since for any $x \in [0, \infty)$,
\begin{equation}
\label{lawordering}
\lim_{r \to \infty} \mathbb{P}_u (\tau_v^{\mathrm{low}} > x) 
\leq \lim_{r \to \infty} \mathbb{P}_u (\tau_v^{\mathrm{int}} > x) 
\leq \lim_{r \to \infty} \mathbb{P}_u (\tau_v^{\mathrm{upp}} > x).
\end{equation}
\end{proof}

\appendix
\renewcommand*{\thesection}{\Alph{section}}

\section{Appendix: the input process}
\label{APP.a}
The main target of this appendix is to prove Proposition~\ref{inputprocess} in Section~\ref{S2}. 
We use path large-deviation techniques. For simplicity, we suppress the index for the arrival 
rates $\lambda_U$ and $\lambda_V$, and consider a general rate $\lambda$. We show 
that with high probability the input process lies in a narrow tube around the deterministic path 
$t \mapsto (\lambda/\mu)t$.

Consider a single queue, and for simplicity suppress its index. For $T > 0$, define the scaled process
\begin{equation}
Q^+_n(t) = \frac{1}{n} Q^+(nt) = \frac{1}{n} \sum_{j =1}^{N(nt)} Y_j, \qquad t \in [0,T],
\end{equation}
with $Q^+_n(0) = 0$. We have
\begin{equation}
\mathbb{E}[Q^+_n(t)] = \frac{1}{n} \frac{\lambda nt}{\mu} = \frac{\lambda}{\mu} t,
\end{equation}
and, by the strong law of large numbers, $Q^+_n(t) \to (\lambda/\mu)t$ almost surely 
for every $t$ as $n \to \infty$.

When studying the process $t \mapsto Q^+_n(t)$, we need to take into account that this is a 
combination of the Poisson arrival process $t \mapsto N(nt)$ and the exponential service times $Y_j$, $j\in\mathbb{N}$. Two different types of fluctuations can occur: packets arrive at a slower/faster rate than $\lambda$, 
respectively, service times for each packet are shorter/longer than their mean $1/\mu$. Both need to be considered for a 
proper large-deviation analysis. 


\subsection{LDP for the two components}

\begin{definizione}[{\bf Space of paths}]
Consider the space $L_{\infty}([0,T])$ of \textit{essentially bounded} functions in $[0,T]$, with the norm 
$\|f||_{\infty} = \text{ess} \sup_{x \in [0,T]} |f(x)|$ called the essential norm. A function $f$ is essentially 
bounded, i.e., $f \in L_{\infty}([0,T])$, when there is a measurable function $g$ on $[0,T]$ such that 
$f=g$ except on a set of measure zero and $g$ is bounded. Let $\mathcal{AC}_T\subset L_{\infty}([0,T])$ 
denote the space of \textit{absolutely continuous} functions $f\colon\, [0,T] \to \mathbb{R}$ such that 
$f(0)=0$. 
\hfill\qed
\end{definizione}

\noindent
Given the Poisson arrival process $t \mapsto N(nt)$ with rate $\lambda$, define the scaled process 
$t \mapsto Z_n(t)$ by
\begin{equation}
Z_n(t) =  \frac{1}{n} N(nt) = \frac{1}{n} \sum_{i=1}^{[nt]}X_i, \qquad t \in [0,T],
\end{equation}
where $X_i = \mathrm{Poisson}(\lambda)$ are i.i.d.\ random variables. Note that $N(nt) = \mathrm{Poisson}
(\lambda n t)$. Let $\nu_n$ be the law of $(Z_n(t))_{t \in [0,T]}$ on $L_{\infty}([0,T])$. Note that $Z_n(t)$ is 
asymptotically equivalent to $N(t)$ with mean $\mathbb{E}[Z_n(t)] = \lambda t$, and $(Z_n(t))_{t \in [0,T]} 
\to (\lambda t)_{t \in [0,T]}$ as $n \to \infty$.

We recall the definition of large deviation principle (LDP).
\begin{definizione}[{\bf Definition of large deviation principle}]
A family of probability measures $(P_n)_{n \in \mathbb N}$ on a Polish space $\mathcal X$ is said to satisfy the 
\emph{large deviation principle (LDP)} with rate $n$ and with good rate function $I\colon\mathcal X \to [0,\infty]$ if
\begin{equation}  
\begin{array}{lll}
&\forall C \subset \mathcal X \text{ closed}\colon &\limsup_{n\to\infty} \frac1n \log P_t(C) \leq - I(C),\\
&\forall O \subset \mathcal X \text{ open}\colon &\liminf_{n\to\infty} \frac1n \log P_t(O) \geq - I(O),\\
 \end{array}
\end{equation}
where $I(S) = \inf_{x \in S} I(x)$, $S \subset \mathcal X$. A good rate function satisfies: 
(1) $I \not\equiv \infty$, (2) $I$ is lower semi-continuous, (3) $I$ has compact level sets. 
\hfill\qed
\end{definizione} 
We begin by stating the LDP for the arrival process $(Z_n(t))_{t \in [0,T]}$.

\begin{lemma}[{\bf LDP for the arrival process}]
The family of probability measures $(\nu_n)_{n\in\mathbb{N}}$ satisfies the LDP on $L_{\infty}([0,T])$ 
with rate $n$ and with good rate function $I_N$ given by
\begin{equation}
I_N(\eta) = 
\left\{\begin{array}{ll} 
\int_0^T \Lambda_N^*(\dot{\eta}(t))\,dt, &\eta \in \mathcal{AC}_T,\\ 
\infty,  &\text{otherwise}, 
\end{array}
\right.
\end{equation} 
where $\Lambda_N^*(x) = x \, \log(x/\lambda) - x + \lambda$, $x\in\mathbb{R}_+$.
\end{lemma}

\begin{proof}
Apply Mogulskii's theorem (\cite[Theorem 5.1.2]{DZ98}). Use that $\Lambda_N^*$ is the 
Fenchel-Legendre transform of the cumulant generating function $\Lambda$ defined by 
$\Lambda(\theta) = \log\mathbb{E}(e^{\theta X_1})$, $\theta \in \mathbb{R}$.
\end{proof}

For $\Gamma \subset L_{\infty}([0,T])$, define $I_N(\Gamma) = \inf_{\eta 
\in \Gamma} I_N(\eta)$. The LDP implies that if $\Gamma \subset L_{\infty}([0,T])$ is an 
$I_N$-continuous set, i.e., $I_N(\Gamma) = I_N(\mathrm{int}(\Gamma)) = I_N(\mathrm{cl}
(\Gamma))$, then
\begin{equation}
\lim_{n \to \infty} \frac{1}{n} \log \mathbb{P}\big(Z_n([0,T]) \in \Gamma\big) 
= - I_N(\Gamma).
\end{equation}
Informally, the LDP reads as the approximate statement
\begin{equation}
\mathbb{P}\big(Z_n([0,T]) \approx \eta([0,T]\big) = e^{-n I_N(\eta)[1+o(1)]}, \qquad n \to \infty, 
\end{equation}
where $\approx$ stands for close in the essential norm. Informally, on this event we may approximate
\begin{equation}
\label{Q+sums}
Q^+_n(t) = \frac{1}{n} \sum_{j=1}^{N(nt)} Y_j = \frac{1}{n} \sum_{j=1}^{n Z_n(t)} Y_j 
\approx \frac{1}{n} \sum_{j=1}^{n \eta(t)} Y_j \approx \frac{1}{n}  \sum_{j=1}^{[n \eta(t)]} Y_j,
\qquad t \in [0,T],
\end{equation}
where $\approx$ now stands for close in the Euclidean norm. Given $\eta \in L_{\infty}([0,T])$, 
let $\mu^\eta_n$ denote the law of the last sum in \eqref{Q+sums}. Below we state the LDP for the input process conditional on the arrival process.

\begin{lemma}[{\bf LDP for the input process conditional on the arrival process}]
Given $\eta \in L_{\infty}([0,T])$, the family of probability measures $(\mu^\eta_n)_{n\in\mathbb{N}}$ 
satisfies the LDP on $L_{\infty}([0,T])$ with rate $n$ and with good rate function $I_Q^\eta$ given by
\begin{equation}
I_Q^\eta(\phi) = 
\left\{\begin{array}{ll} 
\int_0^T \Lambda_Q^*\left(\frac{d \phi(t)}{d\eta(t)}\right)\,d\eta(t), 
&\phi \in \mathcal{AC_T},\\
\infty, &\text{otherwise}, 
\end{array}
\right.
\end{equation}
where $\Lambda_Q^*(x) = x \mu - 1 - \log(x \mu)$, $x\in\mathbb{R}_+$.
\end{lemma}

\begin{proof}
Again apply Mogulskii's theorem, this time with $\eta(t)$ as the time index. Use that $\Lambda^*$ 
is the Fenchel-Legendre transform of the cumulant generating function $\Lambda$ defined by 
$\Lambda(\theta) = \log \mathbb{E}(e^{\theta Y_1})$, $\theta \in \mathbb{R}$.
\end{proof}


\subsection{Measures in product spaces}
 
The rate function $I_Q^\eta$ describes the large deviations of the sequence of processes 
$(Q_n^+(t))_{t \in [0,T]}$ given the path $\eta$. To derive the LDP averaged over 
$\eta$, we need a small digression into measures in product spaces.

\begin{definizione}[{\bf Product measures}]
Define the family of probability measures $(\rho_n)_{n\in\mathbb{N}} $ such that $\rho_n
= \nu_n \mu_n^{\eta}$. These measures are defined on the product space $L_{\infty}([0,T]) 
\times L_{\infty}([0,T])$ given by the Cartesian product of the space $L_{\infty}([0,T])$ 
with itself, equipped with the product topology.
\hfill\qed
\end{definizione}

\noindent
The open sets in the product topology are unions of sets of the form $U_1 \times U_2$ with 
$U_1, U_2$ open in $L_{\infty}([0,T])$. Moreover, the product of base elements of $L_{\infty}([0,T])$ 
gives a basis for the product space $L_{\infty}([0,T]) \times L_{\infty}([0,T])$. Define the projections 
$\Pr_i\colon\,L_{\infty}([0,T]) \times L_{\infty}([0,T]) \to L_{\infty}([0,T])$, $i = 1,2$, onto the first and 
the second coordinates, respectively. The product topology on $L_{\infty}([0,T]) \times L_{\infty}([0,T])$ 
is the topology generated by sets of the form $\Pr_i^{-1}(U_i)$, $i=1,2$, where and $U_1, U_2$ are 
open subsets of $L_{\infty}([0,T])$. 

\begin{lemma}[{\bf Product LDP}]
The family of probability measures $(\rho_n)_{n\in\mathbb{N}}$ satisfies the LDP on 
$L_{\infty}([0,T]) \times L_{\infty}([0,T])$ with rate $n$ and with good rate function $I$ given by
\begin{equation} 
\label{pathcost}
I(\phi, \eta) = 
\left\{\begin{array}{ll}
\int_0^T \Lambda_Q^*\left(\frac{d \phi(t)}{d\eta(t)}\right)\, d\eta(t) 
+\int_0^T \Lambda_N^*(\dot{\eta}(t))\,dt, &\phi,  \eta\in \mathcal{AC}_T, \\ 
\infty, &\text{otherwise}.
\end{array}
\right.
\end{equation}
\end{lemma}

\begin{proof}
This follows from standard large deviation theory (see \cite{DZ98}).
\end{proof}


\subsection{LDP for the input process}

The Contraction Principle allows us to derive the LDP averaged over $\eta$. Indeed, 
let $\mathcal{X} = L_{\infty}([0,T]) \times L_{\infty}([0,T])$ and $\mathcal{Y} = L_{\infty}([0,T]) $, 
let $(\rho_n)_{n\in\mathbb{N}}$ be a sequence of product measures on $\mathcal{X}$, 
and consider the projection $\Pr_1$ onto $\mathcal{Y}$, which is a continuous map. Then the 
sequence of induced measures $(\mu_n)_{n\in\mathbb{N}}$ given by $\mu_n = \rho_n \Pr_1^{-1}$ 
satisfies the LDP on $L_{\infty}([0,T])$ with good rate function
\begin{equation}
\tilde{I}_Q(\phi) = \inf_{(\phi, \eta) \in Pr_1^{-1}(\{\phi\})} I(\phi, \eta) 
= \inf_{\eta \in L_{\infty}([0,T])} I(\phi, \eta).
\end{equation}

We can now state the LDP for the input process $(Q_n^+(t))_{t \in [0,T]}$.

\begin{proposizione}[{\bf LDP for the input process}]
\label{A.6}
The family of probability measures $(\mu_n)_{n\in\mathbb{N}}$ 
satisfies the LDP on $L_{\infty}[0,T]$ with rate $n$ and with good rate function $\hat{I}$ given by
\begin{equation}
\hat{I}_Q(\Gamma) = \inf_{\phi \in \Gamma} \tilde{I}_Q(\phi).
\end{equation}
In particular, if $\Gamma$ is $\hat{I}_Q$-continuous, i.e., $\hat{I}_Q(\Gamma) 
= \hat{I}_Q(\mathrm{int}(\Gamma)) = \hat{I}_Q(\mathrm{cl}(\Gamma))$, then
\begin{equation} 
\label{pathingamma}
\lim_{n \to \infty} \frac{1}{n} \log \mathbb{P}\big(Q_n^+([0,T]) \in \Gamma\big) 
= - \hat{I}_Q(\Gamma).
\end{equation}
\end{proposizione}

\begin{proof}
This follows from the Contraction Principle (see \cite{DZ98}). 
\end{proof}

It is interesting to look at a specific subset of $L_{\infty}([0,T])$ that gives good bounds for the 
input process. We are now in a position to prove Proposition~\ref{inputprocess}.

\begin{proof}
Computing the Fenchel-Legendre transforms $\Lambda_Q^*$ and $\Lambda_N^*$, and picking 
$\eta(t) = \lambda t$ and $\phi(t) = (1/\mu) \eta(t) = (1/\mu) \lambda t$, we easily check 
that the rate function attains its minimal value zero. Hence with high probability the input process 
is close to this deterministic path.

We can now estimate the probability of the scaled input process to go outside $\Gamma_{T,\delta}$, 
which represents a tube of width $2 \delta$ around the mean path in the interval $[0,T]$. More 
precisely,
\begin{equation}
\Gamma_{T,\delta} = \left\{\gamma \in L_{\infty}([0,T])\colon\,
\frac{\lambda}{\mu}t  - \delta  < \gamma(t) < \frac{\lambda}{\mu}t + \delta\,\,\,\forall\, t \in [0,T]\right\}.
\end{equation}
We may set $T=1$ for simplicity and look at the scaled input process in the time 
interval $[0,1]$. We have 
\begin{equation}
\hat{I}_Q((\Gamma_{1,\delta})^c) = \hat{I}_Q(\text{int}((\Gamma_{1,\delta}))^c) 
= \hat{I}_Q(\text{cl}((\Gamma_{1,\delta}))^c).
\end{equation} 
Hence $(\Gamma_{1,\delta})^c$ is $\hat{I}_Q$-continuous, and so according to \eqref{pathingamma},
\begin{equation}
\lim_{n \to \infty} \frac{1}{n} \log \mathbb{P}\big(Q_n^+([0,1]) \notin \Gamma_{1,\delta}\big) 
= - \hat{I}_Q((\Gamma_{1,\delta})^c).
\end{equation}
Since
\begin{equation}
\begin{aligned}
& \lim_{n \to \infty} \frac{1}{n} \log \mathbb{P}\big(Q_n^+([0,1]) \notin \Gamma_{1,\delta}\big) \\
&= \lim_{n \to \infty} \frac{1}{n} \log\mathbb{P}\bigg( \bigg\{\frac{\lambda}{\mu}t  
- \delta  < Q^+_n(t) < \frac{\lambda}{\mu}t + \delta \,\,\,\forall\, t \in [0,1]\bigg\}^c\bigg) \\ 
&= \lim_{S \to \infty} \frac{1}{S} \log \mathbb{P}\bigg( \bigg\{\frac{\lambda}{\mu}s  
- \delta S < Q^+(s) < \frac{\lambda}{\mu}s + \delta S \,\,\,\forall\, s \in [0,S] \bigg\}^c \bigg),
\end{aligned}
\end{equation}
where we put $s= nt$ and $S= n$, we conclude that the probability to go out of 
$\Gamma_{S, \delta S}$ is
\begin{equation}
\mathbb{P}\bigg( \bigg\{\frac{\lambda}{\mu}s  - \delta S < Q^+(s) 
< \frac{\lambda}{\mu}s + \delta S\,\,\, \forall\, s \in [0,S]\bigg\}^c\bigg) 
= e^{-S\,\hat{I}_Q((\Gamma_{1,\delta})^c)\, [1+o(1)]}, 
\qquad S\to \infty.
\end{equation}

Because $I_Q$ is convex, to compute $\hat{I}_Q((\Gamma_{1,\delta})^c)$ it suffices to minimise 
over the linear paths. The minimizer turns out to be one of the two linear paths that go from the 
origin $(0,0)$ to $(1,\lambda/\mu \pm\delta)$, i.e., $\gamma^*(t) =k t$ with $k =  
(\lambda \pm \delta \mu)/\mu$. By construction, $\hat{I}_Q((\Gamma_{1,\delta})^c) = \tilde{I}_Q(\gamma^*) 
= \inf_{\eta \in L_{\infty}([0,1])} I(\gamma^*,\eta)$, where 
\begin{equation}
I(\gamma^*, \eta) = \int_0^1 \Lambda_Q^*\left(\frac{d \gamma^*(t)}{d\eta(t) }\right)\, d\eta(t) 
+ \int_0^1 \Lambda_N^*(\dot{\eta}(t))\,dt.
\end{equation}
We want to minimize the sum over all paths $\eta$ such that $\eta(0)=0$. Both integrals are convex 
as a function of $\gamma^*$ and $\eta$, hence they are minimized by linear paths. Our choice of 
$\gamma^*(t) = kt$ is linear, so we set $\eta(t) = ct$ with some constant $c >0$. We can then 
write
\begin{equation}
\begin{aligned}
I(\gamma^*, \eta)  
&= \int_0^{\eta(1)} \Lambda_Q^*\left(\frac{d \gamma^*(t)}{c dt}\right)\,c dt 
+ \int_0^1 \Lambda_N^*(\dot{\eta}(t))\,dt \\ 
&= \int_0^c \Lambda_Q^*\bigg(\frac{k}{c}\bigg)\,c dt + \int_0^1 \Lambda_N^*(c)\, dt \\ 
&= c\bigg[\frac{k}{c} \mu - 1 - \log\bigg(\frac{k\mu}{c}\bigg)\bigg] 
+ c \log \left(\frac{c}{\lambda }\right) -c + \lambda.
\end{aligned}
\end{equation}
The value of $c$ that minimizes the right-hand side is $c=\sqrt{\lambda k \mu}$. Substituting this 
into the formula above, we get
\begin{equation}
K_{\delta} = \tilde{I}_Q(\gamma^*) = k \mu + \lambda  - 2 \sqrt{\lambda k \mu} 
= (\lambda + \delta \mu) + \lambda - 2 \sqrt{\lambda(\lambda + \delta \mu)}.
\end{equation}
Note that $K_{\delta}>0$ for all $\delta >0$ and $\lim_{\delta \downarrow 0} K_{\delta} = 0$. This proves 
the claim.
\end{proof}


\section{Appendix: the output process}
\label{APP.b}

The main goal of this appendix is to prove Proposition~\ref{outputprocess} in Section~\ref{S2}. 
In Section~\ref{APP.b1} we show a lower bound for the output process for the nodes in $U$, in a 
setting where the nodes in $U$ are not influenced by the nodes in $V$. We study the system 
up to time $T_U$. 

In Section~\ref{APP.b2} we show that, until the pre-transition time, the system in the internal model 
behaves actually as we described.


\subsection{The output process in the isolated model}
\label{APP.b1}

Recall that in the isolated model a node in $U$ keeps activating and deactivating independently of 
the nodes in $V$, until its queue length hits zero. We again consider a single queue for a node in 
$U$ and for simplicity suppress its index. In order to show that the output process $t \mapsto 
Q^-(t) = c T(t)$ when properly rescaled is close to a deterministic path with high probability, we will 
provide a lower bound for the output process. The upper bound $Q^-(t) \leq ct$ is trivial and holds 
for any $t \geq 0$, by the definition of output process.

\begin{lemma}[{\bf Auxiliary output process}] 
\label{lemma1}
For all $\delta > 0$ and $T$ large: 
\begin{itemize}
\item[(i)] 
With high probability the process 
\begin{equation}
Q^{\mathrm{LB},T}(t) = \gamma_U r + \rho_U t - \delta T - ct, \qquad t \in [0,T],
\end{equation}
is a lower bound for the actual queue length process $(Q(t))_{t \in [0,T]}$.
\item[(ii)] 
The probability of the lower bound in (i) failing is
\begin{equation} 
\label{problemma1}
\frac{1}{2} e^{-K_{\delta} T \,[1+o(1)]}, \qquad T \to \infty,
\end{equation} 
with $K_{\delta}=(\lambda + \delta \mu) + \lambda - 2 \sqrt{\lambda(\lambda + \delta \mu)}$.
\end{itemize}
\end{lemma}

\begin{proof}
(i) By Proposition~\ref{inputprocess}, with high probability we have $Q^+(t) \geq \rho_U t- \delta T$ 
for any $\delta >0$. Trivially, $Q^-(t) \leq ct$. It is therefore immediate that with high probability 
$Q^{\mathrm{LB},T}(t) \leq Q(t)$.\\
(ii) The exponentially small probability of $Q^+(t)$ going below the lower bound is half of the 
probability given by Proposition~\ref{inputprocess}, i.e.,
\begin{equation}
\frac{1}{2} e^{-K_{\delta} T \, [1+o(1)]}, \qquad T \to \infty,
\end{equation}
with $K_{\delta}=(\lambda + \delta \mu) + \lambda - 2 \sqrt{\lambda(\lambda + \delta \mu)}$.
\end{proof}

We study the system up to time $T_U$ defined in Definition~\ref{defTu}, the expected time a single 
node queue takes to hit zero. We will prove in Appendix~\ref{APP.b2} that the pre-transition time in 
the internal model with high probability coincides in distribution with the pre-transition time in the 
isolated model, which occurs with high probability before $T_U$. Hence it is enough to study the 
isolated model up to $T_U$.

\begin{definizione}[{\bf Auxiliary times}]
We next define two times that will be useful in our analysis.
\begin{itemize}
\item[($T_U^*$)]
Consider the auxiliary output process $Q^{\mathrm{LB},T_U}(t)$ up to time $T_U$. We define 
$T^*_U$ as the time needed for the process to hit zero, i.e.,
\begin{equation}
T^*_U = T^*_U(r) = \frac{\gamma_ur - \delta T_U }{c-\rho_U} 
= \frac{\gamma_u - \delta \alpha}{c-\rho_U}r = \alpha' r \asymp r,
\end{equation}
with $\alpha' = \frac{\gamma_u - \delta \alpha}{c-\rho_U}$. The difference $T_U - T_U^* 
= \frac{\delta \alpha}{c-\rho_U}r$ is of order $r$. The queue length at time $T^*_U$ is not 
zero, but still of order $r$. 
\item[($T_U^{**}$)]
We define a smaller time $T_U^{**}$ such that, not only $Q(T_U^{**}) \asymp r$, but also 
$Q^{\mathrm{LB},T_U}(T_U^{**}) \asymp r$, i.e.,
\begin{equation}
T_U^{**}= T_U^{**}(r)= T_U - 2(T_U - T^*_U) 
= \bigg(\frac{\gamma_U - 2 \delta \alpha}{c- \rho_U}\bigg) r = \alpha'' r \asymp r,
\end{equation}
with $\alpha''=\frac{\gamma_U - 2 \delta \alpha}{c- \rho_U}$.
\end{itemize}
\hfill\qed
\end{definizione}

\begin{definizione}[{\bf Inactivity process}]
Define the \textit{inactivity process} by setting $W(t) = t- T(t)$, which equals the total 
amount of inactivity time until time $t$. 
\hfill\qed
\end{definizione}

\noindent
Recall that the service process $t \mapsto Q^-(t)$ with $Q^-(0)=0$ is an alternating sequence 
of activity periods and inactivity periods. The activity periods $Z_i$, $i\in\mathbb{N}$, are i.i.d.\ 
exponential random variables with mean 1. The inactivity periods $W_m$, $m\in\mathbb{N}$, 
are exponential random variables with a mean that depends on the actual queue length at the 
time when each of these periods starts, namely, if $W_m = \big[t_m^{(i)}, t_m^{(f)}\big]$, then 
$W_m \simeq \mathrm{Exp}(g_U(Q(t_m^{(i)}))+O(1/r))$. The queue length during this inactivity 
intervals is actually increasing, but we are considering very small intervals, whose lengths are 
of order $1/r$, so that the queue length does not change much and the error is then 
$O(1/r)$.

To state our lower bound on the output process, we need the following two lemmas.

\begin{lemma}[{\bf Upper bound on number of activity periods}]
 \label{lemma2}
Let $M(t)$ be the number of activity periods that end before time $t$. Then, 
for all $\epsilon_1 >0$ and $r$ large:
\begin{itemize}
\item[(i)]
With high probability
\begin{equation}
M(T^{**}_U) \leq (1+ {\epsilon}_1) T^{**}_U.
\end{equation} 
\item[(ii)]
The probability of the upper bound in (i) failing is 
\begin{equation} 
\label{problemma2}
e^{-K_1 r \, [1+o(1)]} + \frac{1}{2} e^{-K_{\delta} \alpha r \, [1+o(1)]}, \qquad r \to \infty,
\end{equation}
with $K_1 = \alpha'' \frac{\epsilon_1 - \log(1+ \epsilon_1)}{1+\epsilon_1}$, $K_{\delta}$ 
as in Lemma~\ref{lemma1}
\end{itemize}
\end{lemma}

\begin{proof}
(i) Note that $M(T^{**}_U)$ counts the number of activity periods before time $T^{**}_U$, each 
of which has an average duration $1$. Since activity periods alternate with inactivity periods, 
we expect $M(T^{**}_U)$ to be less than $T^{**}_U$. Assume now, for small $\epsilon_1 > 0$, 
that $M(T_U^{**}) > (1+\epsilon_1) T_U^{**}$, which means that the number of activity periods 
before $T_U^{**}$ is greater than the length of the interval $[0,T_U^{**}]$. This implies that the 
average length of each activity period before time $T_U^{**}$ is strictly less than $1$, namely, 
that $\frac{1}{T_U^{**}} \sum_{i=1}^{T_U^{**}} Z_i \leq 1/(1+\epsilon_1)$.  According to 
Cram\'er's theorem, we can compute the probability of this last event as
\begin{equation} 
\label{probM}
\mathbb{P} \left(\sum_{i=1}^{T^{**}_U} Z_i 
\leq \bigg(\frac{1}{1+\epsilon_1}\bigg)T^{**}_U\right) 
= e^{-T^{**}_U I\big(\frac{1}{1+\epsilon_1}\big) \, [1+o(1)]}, \qquad r \to \infty,
\end{equation}
with rate function $I(x) = x \log(x) -x +1$. Therefore, it occurs with exponentially small probability. 
Hence $M(T_U^{**}) > (1+\epsilon_1) T_U^{**}$ must also occur with a 
probability which is also exponentially small. With high probability we 
then have that
\begin{equation}
M(T^{**}_U) \leq (1+ {\epsilon}_1) T^{**}_U.
\end{equation} 
Recall that $T_U^{**} = \alpha'' r$. The counting of alternating activity and inactivity periods 
gets affected when the queue length hits zero, since then the node is forced to switch itself
off and the lengths of the activity periods are not regular anymore. Since at time $T_U^{**}$ 
with high probability the queue length is still of order $r$, the probability that it hits zero at 
any time in the interval $[0,T_U^{**}]$ is very small, since this event would imply the node 
to have a queue length that is below the lower bound, $Q(T_U^{**}) \leq Q^{\mathrm{LB},
T_U}(T_U^{**}) = \gamma_U r + \rho_U T_U^{**} - \delta T_U - c T_U^{**}$, which happens 
with an exponentially small probability by Lemma~\ref{lemma1}.

(ii) We can write
\begin{equation}
\begin{split}
\mathbb{P}(M(T_U^{**}) > (1+\epsilon_1) T_U^{**}) 
& \leq e^{-T^{**}_U I\big(\frac{1}{1+\epsilon_1}\big) \, 
[1+o(1)]} + \frac{1}{2} e^{-K_{\delta} T_U \, [1+o(1)]} \\
& \, =  e^{-K_1 r \, [1+o(1)]} + \frac{1}{2} e^{-K_{\delta} \alpha r \, [1+o(1)]}, 
\qquad r \to \infty,
\end{split}
\end{equation}
with $K_1 = \alpha'' I\big(\frac{1}{1+\epsilon_1}\big) = \alpha'' \frac{\epsilon_1 
- \log(1+ \epsilon_1)}{1+\epsilon_1}$, $K_{\delta}$ as in Lemma~\ref{lemma1}.
\end{proof}

\begin{lemma}[{\bf Upper bound on inactivity process}]
\label{lemma3}
For all $\delta, \epsilon_1, \epsilon_2 > 0$ small and $r$ large: 
\begin{itemize}
\item[(i)]
With high probability
\begin{equation}
W(T^{**}_U) \leq \epsilon_2 r.
\end{equation}
\item[(ii)]
The probability of upper bound in (i) failing is 
\begin{equation} 
\label{problemma3}
\begin{aligned}
&\mathbb{P}\big(W(T^{**}_U) > \epsilon_2r\big) 
\leq e^{-K_{\delta} \alpha r \,[1+o(1)]} + e^{-K_1 r \,[1+o(1)]}\\ 
&\qquad + e^{-\big(K_2 r + K_3 \frac{r}{g_U(r)}+K_4 r \log g_U(r)\big) \,[1+o(1)]}, 
\qquad r \to \infty,
\end{aligned}
\end{equation}
with $ K_2 = \alpha'' (1+\epsilon_1) 
\big(- 1 - \log \big(\frac{\epsilon_2}{ \alpha'' (1+\epsilon_1)}\big)\big), 
K_3 =\epsilon_2, K_4 = \alpha'' (1+\epsilon_1)$.
\end{itemize}
\end{lemma}

\begin{proof}
(i) Since $M(t)$ counts the number of activity periods, and we start with an active node 
(in the starting configuration $u$ all nodes in $U$ are active), we have
\begin{equation}
W(T^{**}_U) \leq \sum_{m=1}^{M(T^{**}_U)} W_m \leq \sum_{m=1}^{M(T^{**}_U)} \hat{W}_m, 
\end{equation}
where $\hat{W}_m$ are i.i.d.\ exponential random variables with rate $g_U(Q^{\mathrm{LB},T_U}
(T^{***}_U))$, and $T^{***}_U$ is the starting point of the last inactivity period before time 
$T^{**}_U$. By the construction of $T_U^{**}$, we know that $Q^{\mathrm{LB},T_U}(T_U^{***})$ 
is of order $r$. The last inactivity period is expected to be longer than the previous ones, since 
the rates depend on the actual queue length, which is decreasing in time. To make the inactivity 
periods $\hat{W}_m$ longer, we consider the lower bound $Q^{\mathrm{LB},T_U}(t)$ for the 
actual queue length given in Lemma~\ref{lemma1}. 

By Lemma~\ref{lemma2}, with high probability $M(T^{**}_U) \leq (1+ {\epsilon}_1) T^{**}_U$, 
and so
\begin{equation} 
\label{eq1}
W(T^{**}_U) \leq \sum_{m=1}^{M(T^{**}_U)} \hat{W}_m 
\leq \sum_{m=1}^{(1+ {\epsilon}_1) T^{**}_U} \hat{W}_m.
\end{equation} 
Define $n = [(1+\epsilon_1)T^{**}_U]$. By Cram\'er's theorem, for small $\epsilon_3 >0$,
\begin{equation} 
\label{eq2}
\begin{split}
\mathbb{P}\left(\sum_{m=1}^{(1+ {\epsilon}_1) T^{**}_U} \hat{W}_m \geq \epsilon_3 T^{**}_U \right) 
&\leq \mathbb{P}\left(\sum_{m=1}^n \hat{W}_m \geq \frac{\epsilon_3}{1+\epsilon_1} n \right) \\
&=  e^{-nI\big(\frac{\epsilon_3}{1+ \epsilon_1}\big)\,[1+o(1)]} 
= e^{-T^{**}_U(1+\epsilon_1) I\big(\frac{\epsilon_3}{1+ \epsilon_1}\big)\,[1+o(1)]}, \qquad n \to \infty,
\end{split}
\end{equation}
where $I$ is the rate function given by
\begin{equation}
I(x) = \frac{x}{g_U(Q^{\mathrm{LB},T_U}(T^{***}_U))} -1 -\log x 
+ \log g_U(Q^{\mathrm{LB},T_U}(T^{***}_U)).
\end{equation}
In order to apply Cram\'er's theorem, take $\epsilon_3 > (1+ \epsilon_1)/g_U(Q^{\mathrm{LB},T_U}
(T^{***}_U)) \asymp 1/g_U(r)$ arbitrarily small. Combining \eqref{eq1}--\eqref{eq2}, we obtain that 
with high probability
\begin{equation}
W(T_U^{**}) \leq \epsilon_3 T_U^{**} = \epsilon_3 \alpha'' r= \epsilon_2 r,
\end{equation}
where $\epsilon_2 = \epsilon_3 \alpha''$ can be taken arbitrarily small.\\
(ii) For large $r$,
\begin{equation}
\begin{split}
\mathbb{P}\left(\sum_{m=1}^{(1+ {\epsilon}_1) T^{**}_U} \hat{W}_m > \epsilon_3 T^{**}_U \right) 
\leq & \, e^{-T^{**}_U (1+\epsilon_1) I\big(\frac{\epsilon_3}{1+ \epsilon_1}\big)\,[1+o(1)]} \\
= \, &e^{-\alpha'' r (1+\epsilon_1) \big( \frac{\epsilon_3}{(1+\epsilon_1) 
 g_U(r)} - 1 - \log \big(\frac{\epsilon_3}{1+\epsilon_1}\big) + \log g_U(r)\big)\,[1+o(1)]} \\
= \, & e^{-\big[\alpha'' (1+\epsilon_1) \big(- 1 - \log \big(\frac{\epsilon_3}{1+\epsilon_1}\big)\big)  r 
+ \epsilon_3 \alpha'' \frac{r}{g_U(r)} +\alpha''(1+\epsilon_1) r \log(g_U(r))\big]\,[1+o(1)] }\\
= \, & e^{-\big(K_2 r + K_3 \frac{r}{g_U(r)}+K_4 r \log g_U(r)\big)\,[1+o(1)]}, \qquad r \to \infty,
\end{split}
\end{equation}
where $K_2 = \alpha'' (1+\epsilon_1) 
\left(- 1 - \log \bigg(\frac{\epsilon_2}{\alpha''(1+\epsilon_1)}\bigg)\right),
K_3 =\epsilon_3\alpha''= \epsilon_2,
K_4 = \alpha'' (1+\epsilon_1).$
We also have to consider the probabilities computed in (\ref{problemma1}) and (\ref{problemma2}). 
Hence we have
\begin{equation}
\begin{aligned}
&\mathbb{P}\big(W(T^{**}_U) \leq \epsilon_3 T^{**}_U\big) 
\leq e^{-K_{\delta} \alpha r \,[1+o(1)]} + e^{-K_1 r \,[1+o(1)]}\\ 
&\qquad + e^{-\big(K_2 r + K_3 \frac{r}{g_U(r)}+K_4 r \log g_U(r)\big)\,[1+o(1)]}, 
\qquad r \to \infty.
\end{aligned}
\end{equation}
\end{proof}

We are now in a position to prove Proposition~\ref{outputprocess}.

\begin{proof}
The equation $Q^-(t) \geq ct - \epsilon r$ can be read as $T(t) \geq t - \epsilon r/c$. This is 
equivalent to saying that $W(t) \leq \epsilon r/c $ for all $t \in [0,T_U]$. By taking $\epsilon_2 
=\epsilon/(3c)$ in Lemma~\ref{lemma3}, we know that, for all $t \in [0,T_U^{**}]$, $W(t) 
\leq W(T_U^{**}) \leq \epsilon r/(3c)$. Moreover, in the interval $[T_U^{**}, T_U]$, the 
cumulative amount of inactivity time is trivially bounded from above by the length of the interval, 
which is $\frac{2\delta r}{c-\rho_U} \leq 2 \epsilon r / (3c)$, and $\epsilon$ can be taken 
arbitrarily small, since $\delta$ can be taken arbitrarily small. Putting the two bounds together, we 
find that with high probability 
\begin{equation}
W(t) \leq \epsilon_2 r + \frac{2\delta r}{c-\rho_U} 
\leq  \frac{1}{3} \frac{\epsilon r}{c} + \frac{2}{3} \frac{\epsilon r}{c} = \frac{\epsilon r}{c}, 
\qquad t \in [0,T_U]. 
\end{equation}
It is immediate to see that the probability of this not happening is given by (\ref{problemma3}).
\end{proof}

The above lower bound $Q^-(t) \geq ct- \epsilon r$ and the trivial upper bound $Q^-(t) \leq ct$ 
imply that with high probability the output process $Q^-(t)$ stays close to the path $ c \mapsto ct$ 
by sending $\epsilon$ to zero. In other words, the node stays almost always active all the time 
before $T_U$. 


\subsection{The output process in the internal model}
\label{APP.b2}

In this section we want to couple the isolated model and the internal model and show that they 
have identical behavior in the time interval $[0,\bar{\tau}_v^{\mathrm{int}}]$. Hence it follows 
that the output process in the internal model for nodes in $U$ actually behaves as in the isolated 
model described in Section~\ref{APP.b1}, until the pre-transition time. 

\begin{proposizione}
\label{prop:coincide}
Let $X_i^{\mathrm{int}}(t)$ and $X_i^{\mathrm{iso}}(t)$ denote the activity state of a node $i$ at time 
$t$ in the internal and the isolated model, respectively. Then
\begin{equation}
\lim_{r \to \infty} \mathbb{P}_u \big( X_i^{\mathrm{int}}(t) = X_i^{\mathrm{iso}}(t)\,\,\,
\forall\, i \in U \cup V \,\,\forall\, 
t \in [0,\bar{\tau}_v^{\mathrm{int}}]  \big) = 1.
\end{equation}
Consequently, with high probability the pre-transition times in the internal and the isolated model 
coincide, i.e.,
\begin{equation}
\lim_{r \to \infty} \mathbb{P}_u(\bar{\tau}_v^{\mathrm{int}} = \bar{\tau}_v^{\mathrm{iso}}) = 1.
\end{equation}
\end{proposizione}

\begin{proof}
In Section~\ref{APP.b1} we determined upper and lower bounds for the output process for nodes 
in $U$ in the isolated model up to time $T_U$. Assume now that $\bar{\tau}_v^{\text{int}} \leq T_U$. 
When considering the internal model and the set of nodes in $V$, we immediately see that these 
bounds are not true for the whole interval $[0,T_U]$, since at time $\bar{\tau}_v^{\text{int}}$ 
already some nodes in $V$ start to activate and influence the behavior of nodes in $U$. 

If we look at the interval $[0,\bar{\tau}_v^{\mathrm{int}}]$, then we note that the queue length process 
for a node $i \in U$ is not affected by nodes in $V$, and so it behaves in exactly the same way as if 
the node were isolated. The activation and deactivation Poisson clocks at node $i$ are synchronized, 
and are ticking at the same time in the isolated model and in the internal model, so that 
$X_i^{\mathrm{int}}(t) = X_i^{\mathrm{iso}}(t)$. Moreover, the activity states of nodes in $V$ are always 
equal to 0 in both models. Hence we conclude that the activity states of every node coincide up to 
the pre-transition time $\bar{\tau}_v^{\text{int}}$. Consequently, the pre-transition times in the internal 
and the isolated model coincide on the event $\{ \bar{\tau}_v^{\text{int}} \leq T_U \}$, which can then 
be written as the event $\{ \bar{\tau}_v^{\text{iso}} \leq T_U \}$. For the latter we know that it has 
a high probability as $r \to \infty$ (see proof of Proposition~\ref{proptwiddle} in Section 3).
\end{proof}

\bibliographystyle{model1-num-names}
\nocite{*}
\section*{\refname}
\bibliography{mybib}

\begin{thebibliography}{paper}

\bibitem{BK80}
R.R.\ Boorstyn, A.\ Kershenbaum,
Throughput analysis of multihop packet radio networks,
In: Proc.\ ICC '80 Conf., 1361--1366.

\bibitem{BKMS87}
R.R.\ Boorstyn, A.\ Kershenbaum, B.\ Maglaris, V.\ Sahin,
Throughput analysis in multihop CSMA packet radio networks,
IEEE Trans.\ Commun. 25 (3) (1987), 267--274.

\bibitem{BdHNT17}
S.C.\ Borst, F.\ den Hollander, F.R.\ Nardi, S.\ Taati,
Crossover times in bipartite networks with activity constraints and time-varying switching rates
[arXiv:1912.13011], Preprint 2019.

\bibitem{BBvL14}
N.\ Bouman, S.C.\ Borst, J.S.H.\ van Leeuwaarden, 
Delay performance in random-access networks,
Queueing Systems 77 (2014), 211--242.

\bibitem{BdH05}
A.\ Bovier, F.\ den Hollander,
Metastability, Springer (2005).

\bibitem{DZ98}
A.\ Dembo, O.\ Zeitouni,
Large Deviations Techniques and Applications, second ed.,
Springer, 1998.

\bibitem{DT08}
M.\ Durvy, O.\ Dousse, P.\ Thiran,
Border effects, fairness, and phase transitions in large wireless
networks,
In: Proc.\ IEEE Infocom 2008, 601--609.

\bibitem{DDT07}
M.\ Durvy, O.\ Dousse, P.\ Thiran,
Modeling the 802.11 protocol under different capture and sensing
capabilities.
In: Proc.\ IEEE Infocom 2007, 2356--2560.

\bibitem{DTD09}
O.\ Dousse, P.\ Thiran, M.\ Durvy
On the fairness of large CSMA networks,
IEEE J.\ Sel.\ Areas Commun.\ 27 (2009) 1093--1104.

\bibitem{DT06}
M.\ Durvy, P.\ Thiran,
A packing approach to compare slotted and non-slotted medium access
control.
In: Proc.\ IEEE Infocom 2006.

\bibitem{GSK08}
M.\ Garetto, T.\ Salonidis, E.W.\ Knightly,
Modeling per-flow throughput and capturing starvation in CSMA
multi-hop wireless networks,
IEEE/ACM Trans.\ Netw. 16 (4), 864--877.

\bibitem{GBW14}
J.\ Ghaderi, S.C.\ Borst, P.A.\ Whiting,
Queue-based random-access algorithms: fluid limits and stability issues,
Stochastic Systems 4 (2014), 81--156.

\bibitem{GS10}
J.\ Ghaderi, R.\ Srikant,
On the design of efficient CSMA algorithms for wireless networks.
In: Proc.\ CDC 2010 Conf., 954--959.

\bibitem{dHNT17}
F.\ den Hollander, F.R.\ Nardi, S.\ Taati,
Metastability of hard-core dynamics on bipartite graphs,
Electron.\ J.\ Probab.\ 23 (2018), no. 97, 1--65.

\bibitem{JSSW10}
L.\ Jiang, D.\ Shah, J.\ Shin, J.\ Walrand,
Distributed random access algorithm: scheduling and congestion control,
IEEE Trans.\ Inf.\ Theory 56 (2010), 6182--6207.

\bibitem{Kelly85}
F.P.\ Kelly,
Stochastic models of computer-communication systems,
J.\ Roy.\ Stat.\ Soc.\ B, 47 (3) (1985), 379--395.

\bibitem{KBC87}
A.\ Kershenbaum, R.R.\ Boorstyn, M.-S.\ Chen,
An algorithm for evaluating throughput in multihop packet radio
networks with complex topologies,
IEEE J.\ Sel.\ Areas Commun.\ 5 (6) (1987), 1003-1012.

\bibitem{LKLW09}
S.C.\ Liew, C.H.\ Kai, J.\ Leung, B.\ Wong,
Back-of-the-envelope computation of throughput distributions in CSMA
wireless networks,
In: Proc.\ ICC '09 Conf., 1--19.

\bibitem{NZB16}
F.R.\ Nardi, A.\ Zocca, S.C. \ Borst, 
Hitting time asymptotics for hard-core interactions on grids,
J.\ Stat.\ Phys.\ 162 (2016), 522--576.

\bibitem{OV05}
E.\ Olivieri, M.E.\ Vares,
Large Deviations and Metastability, Cambridge University Press (2005).

\bibitem{PY86}
E.\ Pinsky, Y.\ Yemini,
The asymptotic analysis of some packet radio networks,
IEEE J.\ Sel.\ Areas Commun. 4 (6) (1986), 938--945.

\bibitem{RSS09}
S.\ Rajagopalan, D.\ Shah, J.\ Shin,
Network adiabatic theorem: An efficient randomized protocol for contention resolution,
ACM SIGMETRICS Perf.\ Eval.\ Rev.\ 37 (2009), 133--144.

\bibitem{SS12}
D.\ Shah, J.\ Shin,
Randomized scheduling algorithms for queueing networks,
Ann.\ Appl.\ Prob.\ 22 (2012), 128--171.

\bibitem{WK05}
X.\ Wang, K.\ Kar,
Throughput modeling and fairness issues in CSMA/CA based ad-hoc networks,
In: Proc.\ IEEE Infocom 2005, 23--34.

\bibitem{Yemini83}
Y.\ Yemini,
A statistical mechanics of distributed resource sharing mechanisms.
In: Proc.\ IEEE Infocom '83, 531--539.

\bibitem{Zocca15}
A.\ Zocca,
Spatio-temporal dynamics of random-access networks:
An interacting-particle approach,
Ph.D. thesis Eindhoven University of Technology (2015).

\end{thebibliography}


\end{document}